\documentclass[12pt,reqno,a4paper]{amsart}
\usepackage[utf8]{inputenc}
\usepackage[T1]{fontenc}
\usepackage{amssymb}
\usepackage{pgf,tikz,pgfplots}
\pgfplotsset{compat=1.12}
\usepackage{mathrsfs}
\usetikzlibrary{arrows}
\usepackage{bm}
\usepackage[alpine]{ifsym}
\usepackage{xpicture}
\usepackage{calculator}
\usepackage{graphicx}
\usepackage{enumerate}

\usepackage[a4paper,top=3.3cm,bottom=3cm,left=3cm,right=3cm,bindingoffset=5mm]{geometry}

\def\sideremark#1{\ifvmode\leavevmode\fi\vadjust{\vbox to0pt{\vss % the remark
      \hbox to 0pt{\hskip\hsize\hskip1em           %                will appear only
 \vbox{\hsize2cm\tiny\raggedright\pretolerance10000%                on the side
 \noindent #1\hfill}\hss}\vbox to8pt{\vfil}\vss}}} %
\usepackage{color}

\usepackage[colorlinks=true]{hyperref}
\newtheorem{theorem}{Theorem}[section]
\newtheorem{lemma}[theorem]{Lemma}
\newtheorem{corollary}[theorem]{Corollary}
\newtheorem{proposition}[theorem]{Proposition}

\theoremstyle{definition}
\newtheorem{example}[theorem]{Example}
\newtheorem{remark}[theorem]{Remark}
\newtheorem{definition}[theorem]{Definition}

\numberwithin{equation}{section}

\begin{document}
\title[Newton-Okounkov bodies for valuations of Hirzebruch surfaces]{Seshadri-type constants and Newton-Okounkov bodies for non-positive at infinity valuations of Hirzebruch surfaces}

\author[C. Galindo]{Carlos Galindo}

\address{Universitat Jaume I, Campus de Riu Sec, Departamento de Matem\'aticas \& Institut Universitari de Matem\`atiques i Aplicacions de Castell\'o, 12071
Caste\-ll\'on de la Plana, Spain.}\email{galindo@uji.es}  \email{cavila@uji.es}

\author[F. Monserrat]{Francisco Monserrat}
\address{Instituto Universitario de
Matem\'atica Pura y Aplicada, Universidad Polit\'ecnica de Valencia,
Camino de Vera s/n, 46022 Valencia (Spain).}
\email{framonde@mat.upv.es}

\author[C.-J. Moreno-\'Avila]{Carlos-Jes\'us Moreno-\'Avila}

%\address{Universitat Jaume I, Campus de Riu Sec, Departamento de Matem\'aticas \& Institut Universitari de Matem\`atiques i Aplicacions de Castell\'o, 12071
%Caste\-ll\'on de la Plana, Spain}

%
\subjclass[2010]{Primary: 14C20, 14E15, 13A18}
\keywords{Newton-Okounkov bodies; Flags; Non-positive at infinity valuations}
\thanks{Partially supported by the Spanish Government MICINN/FEDER/AEI/UE, grants PGC2018-096446-B-C22, RED2018-102583-T and BES-2016-076314, as well as by Universitat Jaume I, grant UJI-B2018-10.}

\begin{abstract}
We consider flags $E_\bullet=\{X\supset E\supset \{q\}\}$,
where $E$ is an exceptional divisor defining a non-positive at infinity divisorial valuation $\nu_E$ of a Hirzebruch surface $\mathbb{F}_\delta$ and $X$ the surface given by $\nu_E,$ and determine an analogue of the Seshadri constant for pairs $(\nu_E,D)$, $D$ being a big divisor on $\mathbb{F}_\delta$. The main result is an explicit computation of the vertices of the Newton-Okounkov bodies of pairs $(E_\bullet,D)$ as above, showing that they are quadrilaterals or triangles and distinguishing one case from another. 
\end{abstract}

\maketitle
\section{Introduction}

Let $L$ be a big line bundle of a normal projective complex variety $X$. Consider a real valuation $\nu$ of $X$, that is a valuation of the function field of $X$ centered at the local ring of a closed point in $X$. Assume $H^0(L)\neq 0$ and set $\hat{\mu}_L(\nu)=\lim_{m\to\infty}m^{-1}a_{\rm max}(mL,\nu)$, where $a_{\rm max}(mL,\nu)$ is the last value of the vanishing sequence of $H^0(mL)$ along $\nu$ \cite{BouKurMacSze}. The value $\hat{\mu}_L(\nu)$ contains, for valuations, similar information as the Seshadri constant for points; then we consider it as a Seshadri-type constant for the pair $(L,\nu)$. Seshadri constants were used in \cite{Dem} for studying the Fujita conjecture and other Seshadri-type constants were introduced in \cite{CutEinLaz} for ideal sheaves. The bound $\hat{\mu}_L(\nu)\geq \sqrt{L^2/\text{vol}(\nu)}$, where $\text{vol}(\nu)$ means volume of the valuation $\nu$, is proved in \cite{BouKurMacSze} but the exact value of $\hat{\mu}_L(\nu)$ is, in general, very hard to compute.

A flag of subvarieties of a smooth irreducible complex projective variety $X$ (of dimension $n$) is a sequence of smooth irreducible subvarieties $Y_j$, $0\le j\leq n,$
$$
Y_\bullet:=\{X=Y_0\supset Y_1\supset \ldots \supset Y_n=\{q\}\},
$$ 
where each $Y_j$ has codimension $j$ in $X$. $Y_\bullet$ defines a rank $n$  valuation $\nu_{Y_\bullet}$ of the function field $K(X)$  and the Newton-Okounkov body $\Delta_{\nu_{Y_\bullet}}(D)$ of a big divisor $D$ on $X$ with respect to $\nu_{Y_\bullet}$ (or $Y_\bullet$)  is the closed convex hull of the set 
$$
\bigcup_{m \geq 1}\left \{ \frac{\nu_{Y_\bullet} (f)}{m} \;|\;  f \in H^0(X, {\mathcal O}_{X}(mD))  \setminus \{0\} \right \}.
$$
Newton-Okounkov bodies were introduced by Okounkov \cite{Oko1,Oko2,Oko3} and afterwards developed by Lazarsfeld and Musta{\c t}{\u a} \cite{LazMus} and Kaveh and Khovanskii \cite{KavKho}. These bodies allow us to study linear systems defined by the involved divisor and valuation. As in the case of $\hat{\mu}_L(\nu)$, an explicit computation of these bodies is also very difficult.

Set $p$ a point of the complex projective plane $\mathbb{P}^2=\mathbb{P}^2_\mathbb{C}$. When the flag is $E_\bullet=\{X\supset E\supset \{q\}\}$, $X$ being the rational surface given by a divisorial valuation $\nu_E$ of the fraction field of $\mathcal{O}_{\mathbb{P}^2,p}$ centered at $\mathcal{O}_{\mathbb{P}^2,p}$ defined by the exceptional divisor $E$, $\nu_{E_\bullet}$ is an exceptional curve valuation of the function field of $\mathbb{P}^2$ centered at $\mathcal{O}_{\mathbb{P}^2,p}$. Exceptional curve valuations constitute one of the five classes in the Spivakovsky classification of valuations of function fields of surfaces \cite{Spiv} and its denomination comes from \cite{FavJon1}. The Newton-Okounkov body of a divisor associated with the pull-back of the line bundle $L=\mathcal{O}_{\mathbb{P}^2}(1)$ with respect to $\nu_{E_\bullet}$ has been described, being the Seshadri-type constant $\hat{\mu}_L(\nu_E)$ an important ingredient (see \cite{GalMonMoyNic2} and \cite{CilFarKurLozRoeShr}). This constant has been found useful to treat other important problems. Indeed, $\nu_E$ is called \emph{minimal} when $\hat{\mu}_L(\nu_E)=\sqrt{1/\text{vol}(\nu_E)}$ and a conjecture  strongly involving the above concept and formulated in \cite{GalMonMoy} proves certain evidence in the direction of Nagata's conjecture (see also \cite{DumHarKurRoeSze}). Non-positive at infinity valuations of $\mathbb{P}^2$ constitute an interesting class of divisorial valuations $\nu_E$.  Recently, valuations in this last class  have been studied and used in several contexts \cite{CamPilReg,FavJon3,Mond}. Among their important properties  one can mention that they determine those surfaces given by divisorial valuations of $\mathbb{P}^2$ whose cone of curves is finitely generated and its extremal rays are as few as possible \cite{GalMon}; $\hat{\mu}_L(\nu_E)$ can be explicitly obtained \cite{GalMonMoy}; and the vertices of the Newton-Okounkov body with respect to any valuation $\nu_{E_\bullet}$ as above, where $\nu_E$ is a non-positive at infinity divisorial valuation, can also be explicitly computed \cite{GalMonMoyNic2}. 

In this paper, our basic variety will be $\mathbb{F}_\delta,$ the $\delta$th Hirzebruch surface (for $\delta\geq 0$). For divisorial valuations of these surfaces (that is, those divisorial valuations of the function field  $K(\mathbb{F}_\delta)$ centered at the local ring $\mathcal{O}_{\mathbb{F}_\delta,p}$ of a closed point $p\in \mathbb{F}_\delta$), one can also introduce a concept of non-positivity at infinity, which depends on the value of $\delta$, the position of the point $p$ and certain linear systems (see Definitions \ref{Def_spe_nspe_val} and \ref{Def_NPI_spe_nspe_val}).  As in the case of $\mathbb{P}^2$, these valuations determine those rational surfaces $Z$ defined by divisorial valuations of Hirzebruch surfaces such that the number of generators of their cones of curves are reduced to the minimum possible  \cite{GalMonMor}. Notice that although the  valuations of $\mathbb{F}_\delta$ do no differ from those of $\mathbb{P}^2$ (when they are considered as local objects), the classes of non-positive at infinity valuations of $\mathbb{P}^2$ and $\mathbb{F}_\delta$ are different \cite[Remark 3.10]{GalMonMor}.

The goals of this paper are two-fold. On the one hand, to introduce (and consider) a concept of minimality for divisorial valuations of $\mathbb{F}_\delta$ (Definition \ref{Def_minimal}) and the \emph{computation of the value} $\hat{\mu}_D(\nu)$ for any non-positive at infinity divisorial valuation $\nu$ of $\mathbb{F}_\delta$ and any big divisor $D$ on $\mathbb{F}_\delta$ (see Theorem \ref{Thm_mupico_nonpositive}). Notice that in our situation $\hat{\mu}_D(\nu)=\sup\{t>0\, | \, D^*-tE \text{ is big on }Z\}$  where $D^*$ is the pull-back of $D$ on $Z$ and $E$ the divisor defining $\nu$.
On the other hand, with the help of the previous computation, the \emph{explicit determination} of the vertices of the Newton-Okounkov bodies of big divisors $D$ with respect to  flags $E_\bullet=\{Z\supset E \supset \{q\}\},$ where $Z$ is the rational surface defined by some valuation $\nu$ as above, $E$ the defining divisor of $\nu$ and $D$ the pull-back of a big divisor on $\mathbb{F}_\delta$. Our main results are Theorems \ref{Thm_NOBminimalvaluation}, \ref{Thm_NOBodies_specialcase}, \ref{Thm_NOBcasoespecial_g*=1F}, \ref{Thm_NOBcasoespecial_g*=1M0},  \ref{Thm_NOBodies_nonspecialcase} and \ref{Thm_NOBcasonoespecial_g*=1M1} where, as a consequence of an explicit calculation, we prove that the vertices of our Newton-Okounkov bodies depend only on the expression of $D$, the volume of $\nu$ and the values of the germs at $p$ of the fibre and sections on $\mathbb{F}_\delta$ whose strict transforms (together with those of the exceptional divisors) span the cone of curves. These values are with respect to the two divisorial valuations involved in the exceptional curve valuation $\nu_{E_\bullet}$ (see the paragraph before Definition \ref{Def_spe_nonspe_except_val}).

Section \ref{Sec_Seshadri-typeconstants} introduces the concepts considered in the paper, special and non-special, minimal, non-positive at infinity divisorial valuations, that will be extended to exceptional curve valuations $\nu$ in Section \ref{Sec_NOBFdelta}. We show in Theorem \ref{Thm_NOBminimalvaluation} that minimal with respect to a big divisor $D$ exceptional curve valuations $\nu$ of $\mathbb{F}_\delta$ are those whose Newton-Okounkov body $\Delta_\nu(D)$ is the triangle $T$ given by the truncated convex cone of the $(x,y)$-plane generated by the value semigroup of $\nu$ and bounded by the line $x=\hat{\mu}_D(\nu_r),$ $\nu_r$ being the divisorial valuation defined by the first projection of $\nu$. This fact also happens for valuation of $\mathbb{P}^2$. When $\nu_r$ is not minimal, in our case ($\nu_r$ is non-positive at infinity), $\Delta_\nu(D)$ is either a quadrilateral or a triangle. This last case only happens under certain conditions which depend on the divisor $D$ and the valuation $\nu_r$. Seshadri-type constants and Newton-Okounkov bodies with respect to non-positive at infinity valuations of $\mathbb{P}^2$ can be obtained as a particular case of the results in Sections \ref{Sec_Seshadri-typeconstants} and \ref{Sec_NOBFdelta}.

\section{Seshadri-type constants for non-positive at infinity valuations of Hirzebruch surfaces}\label{Sec_Seshadri-typeconstants}

\subsection{Hirzebruch surfaces and valuations of Hirzebruch surfaces}\label{Subsec_Hirz_mupico}

Let $\mathbb{P}^1=\mathbb{P}_\mathbb{C}^1$ be the projective line over the complex field $\mathbb{C}$ and $\delta$ a non-negative integer. The $\delta$th \emph{Hirzebruch surface} is the projective ruled surface over $\mathbb{P}^1$, $\mathbb{F}_\delta : =\mathbb{P}(\mathcal{O}_{\mathbb{P}^1}\oplus\mathcal{O}_{\mathbb{P}^1}(-\delta))$,  together with the projection morphism $\mathrm{pr}:\mathbb{F}_\delta\to \mathbb{P}^1$. The Picard group Pic$(\mathbb{F}_\delta)$ of $\mathbb{F}_\delta$ is isomorphic to $\mathbb{Z}\oplus\mathbb{Z}$ and admits as generators the class of a fiber $F$ of $\mathrm{pr}$ and that of a section $M$ of $\mathrm{pr}$ linearly equivalent to $\delta F + M_0$ satisfying that $M\cap M_0 = \emptyset$, where $M_0$ denotes, if $\delta>0$ (respectively, $\delta=0$) the unique section on $\mathbb{F}_\delta$ with negative self-intersection (respectively, a section); see for instance \cite[Proposition IV.1]{Beau}. It holds that $F^2=0, F\cdot M=1$ and $M^2=\delta$. 

In the case $\delta>0$, the section $M_0$ is called \emph{special}, and a point $p$ of $\mathbb{F}_\delta$ is \emph{special} if $p\in M_0$ and \emph{general} otherwise. A nef (respectively, big) divisor on $\mathbb{F}_\delta$ is linearly equivalent to $aF + bM$, where $a$ and $b$ are non-negative integers (respectively, $a$ and $b$ are integers such that $b>0$ and $a>-\delta b$ (see \cite[Remark 2.2.27]{Laz1}). 

\medskip

Let $(R,\mathfrak{m})$ be a two-dimensional local regular ring and $K$ its fraction field. A valuation of $K$ is a surjective map $\nu:K^*(=K\setminus \{0\})\to G,$ where $G$ is a totally ordered commutative group, such that, for $f,g\in K^*$, satisfies
$$
\nu (f+g) \geq \min\{\nu(f),\nu(g)\} \text{ and } \nu(f g)=\nu(f) + \nu(g).
$$ 
The local ring $R_\nu=\{f\in K \,|\,\nu(f)\geq 0\}\cup\{0\},$ whose maximal ideal is $\mathfrak{m}_\nu=\{f\in K \,|\, \nu(f)>0\}\cup \{0\}$, is called the \emph{valuation ring} of $\nu$. When it holds that $R\cap\mathfrak{m}_\nu=\mathfrak{m}$, one says that $\nu$ is \emph{centered} at $R$. 

Valuations of $K$ centered at $R$ correspond one-to-one with simple sequences of point blowing-ups 
\begin{equation}\label{Eq_sequencepointblowingups}
\pi: \cdots \rightarrow Z_n\xrightarrow{\pi_n} Z_{n-1}\rightarrow \cdots \rightarrow Z_1 \xrightarrow{\pi_1} Z_0=\text{Spec}R,
\end{equation}
where the first blowing-up $\pi_1$ is at $p:=p_1$  corresponding to the maximal ideal $\mathfrak{m}$ and the blowing-up $\pi_{i+1}$ is centered at the unique closed point $p_{i+1}$ which belongs to the exceptional divisor created by $\pi_i$ such that the valuation is centered at $\mathcal{O}_{Z_i,p_{i+1}}$. The set $\mathcal{C}_\nu=\{p=p_1,p_2,\ldots\}$ is called the \emph{configuration of infinitely near points of} $\nu$. Denote by $E_i$ the exceptional divisor on $Z_i$ obtained by blowing-up $p_i$. A point $p_i$ is \emph{proximate} to $p_j$, denoted by $p_i\to p_j,$ when $p_i$ belongs to the strict transform of $E_j$ on $Z_{i-1}$. The point $p_i$ is called \emph{satellite} when it is proximate to $p_j$, for some $j<i-1$; otherwise, it is named \emph{free}. Given a divisor $D$ on $Z_i$, abusing of notation, we will denote by $\tilde{D}$ and $D^*$ the strict and total transforms of $D$ on any surface $Z_j$ for $j\geq i$; also the strict transforms of the exceptional divisors $E_i$ will be written simply $E_i$. 

The previous valuations were studied by Zariski and Abhyankar (see \cite{Abh1,Abh2,Zar3,ZarSam}). Spivakovsky, in \cite{Spiv}, classifies them in five types according to their dual graphs, which are trees whose vertices correspond $1$-$1$ with the exceptional divisors associated with the  sequence \eqref{Eq_sequencepointblowingups} and two vertices are joined by an edge if the corresponding exceptional divisors intersect. Each vertex of the dual graph is labelled by a positive integer $i$ which represents $E_i$. We say  that two vertices $\alpha$ and $\beta$ satisfy $\alpha \preccurlyeq\beta$ if the path in the dual graph joining $1$ and $\beta$ goes through $\alpha$.   

We are only interested in divisorial and exceptional curve valuations which are two of the types in Spivakovsky's classification. A valuation is \emph{divisorial} when $\mathcal{C}_\nu$ is finite and it is \emph{exceptional curve} (in the terminology of \cite{FavJon1}) if $\mathcal{C}_\nu$ is infinite and there exists a point $p_r\in\mathcal{C}_\nu$ such that $p_i\to p_r$ for all $i>r$. The group $G$ is isomorphic to $\mathbb{Z}$ with the usual ordering (respectively, $\mathbb{Z}^2$ with lexicographical ordering) when the valuation is divisorial (respectively, exceptional curve).
 
Let $\nu$ be a divisorial or exceptional curve valuation of $K$ centered at $R$ and $\mathcal{C}_\nu=\{p_i\}_{i\geq 1}$ its configuration. For each $i\geq 1$, denote by $\mathfrak{m}_i$  the maximal ideal of the local ring $R_i=\mathcal{O}_{Z_i,p_i}$ and set $\nu(\mathfrak{m}_i):=\min\{\nu(x) \ | \ x\in\mathfrak{m}_i\setminus\{0\}\}$. These values satisfy the \emph{proximity equalities} \cite[Theorem 8.1.7]{Cas}: $
\nu(\mathfrak{m}_i)=\sum_{p_j\to p_i}\nu(\mathfrak{m_j}), \ \ i\geq 1,$ whenever the set $\{p_j\in\mathcal{C}_\nu \ | \ p_j\to p_i\}$ is not empty. When $\nu$ is exceptional curve and $p_i\to p_r$ for every $i>r$, then $\nu(\mathfrak{m}_r)=(a,b)$ and $\nu(\mathfrak{m}_i)=(0,c),$ for some $a,b,c\in\mathbb{Z},$ $a,c>0$ \cite{DelGalNun}. 

Divisorial and exceptional curve valuations admit sets of invariants that help to study them, as the sequence of maximal contact values $\{\overline{\beta}_j(\nu)\}_{j=0}^{g+1}$ \cite[(1.5.3)]{DelGalNun} and the sequence of Puiseux exponents $\{\beta_j'(\nu)\}_{j=0}^{g+1}$. Notice that both sequences can be obtained one from other \cite[Theorem 1.11]{DelGalNun}. The set $\{\overline{\beta}_j(\nu)\}_{j=0}^g$ generates the semigroup of values of $\nu,$ $\nu(R\setminus\{0\})\cup \{0\}$. The continued fraction expansions of the values $\{\beta_j'(\nu)\}_{j=0}^{g+1}$ determine (and are  determined by) the dual graph of $\nu$. 

We are interested in geometric results concerning Hirzebruch surfaces and for this reason, from now on $R$ will be the local regular ring $\mathcal{O}_{\mathbb{F}_\delta,p}$, where $\mathbb{F}_\delta$ is a Hirzebruch surface over complex field $\mathbb{C}$ and $p$ a closed point of $\mathbb{F}_\delta$. Along the paper we denote by $\varphi_{C}$ the germ at $p$ of a curve $C$ on $\mathbb{F}_{\delta}$ and by $\varphi_i$ an analytically irreducible germ at $p$ of a curve whose strict transform on $Z_i$ is transversal to $E_i$ at a non-singular point of the exceptional locus. In this case, \emph{valuations of $K$ centered at $R$ will be called valuations of $\mathbb{F}_\delta$.}

\subsection{Seshadri-type constants for non-positive at infinity valuations of Hirzebruch surfaces}\label{Subsection_Seshadri-typeconstantNPI}
In \cite{BouKurMacSze}, the authors consider a vanishing sequence attached to a pair $(L,\nu),$ where $L$ is a big line bundle on a normal projective variety $X$ and $\nu$ is a real valuation of $X$, that is, a real valuation of $K(X)$ centered at the local ring of a closed point of $X$. The value $\lim_{m\to\infty}m^{-1}a_{\rm max}(mL,\nu)$, $a_{\rm max}(mL,\nu)$ being the last value of the above mentioned vanishing sequence, will be denoted $\hat{\mu}_L(\nu)$. When $X=\mathbb{P}^2$, this value encodes for valuations similar information as Seshadri constant for points and we say that $\hat{\mu}_L(\nu)$ is the Seshadri-type constant for the pair $(L,\nu)$. The explicit computation of these constants is a hard work. We devote this subsection to give some details on them when $X$ is a Hirzebruch surface $\mathbb{F}_\delta$ and $\nu$ a divisorial valuation, and to provide its exact value for a large family of divisorial valuations and any big divisor on $\mathbb{F}_\delta$.

Let $\mathbb{F}_\delta$ be a Hirzebruch surface and $p$ a closed point of $\mathbb{F}_\delta$. Let $\nu_n$ be a divisorial valuation of $\mathbb{F}_\delta$ defined by a sequence as \eqref{Eq_sequencepointblowingups} which finishes at $Z_n$. That is, $\nu_n$ is the valuation of the fraction field of $R:=\mathcal{O}_{\mathbb{F}_\delta,p}$ centered at $R$ defined by the exceptional divisor $E_n$. Consider the surface $Z=Z_n$ defined by \eqref{Eq_sequencepointblowingups} when $Z_0=\mathbb{F}_\delta$. According to \cite{EinLazSmi} the volume of $\nu_n$ can be defined as
 
\begin{equation*}
\text{vol}(\nu_n)=\lim_{\alpha\to\infty}\dfrac{\dim_\mathbb{C}(R/\mathcal{P}_\alpha)}{\alpha^2/2},
\end{equation*}
where $\mathcal{P}_\alpha=\{f\in R \ | \ \nu_n(f)\geq \alpha\}\cup\{0\}$. In this case $1/\text{vol}(\nu_n)$  coincides with the last value of the sequence of maximal contact values $\overline{\beta}_{g+1}(\nu_n)$ (see \cite[Remark 2.3]{GalMonMoyNic2}). 

Now consider a pseudoeffective divisor $D{\sim}aF+bM$ on $\mathbb{F}_\delta$, where $\sim$ denotes linear equivalence. $D$ admits a Zariski decomposition $D=P_D+N_D$, where $P_D$ and $N_D$ denote respectively the positive and the negative part of $D$ \cite[Theorem 2.3.19]{Laz1}. When $D$ is nef, then $N_D=0$; if $\delta>0$ and $D$ is big but not nef, then $P_D\sim (b+a/\delta)M$ and $N_D\sim (-a/\delta)M_0$, where $b>0$ and $-b\delta<a<0$. Moreover, the volume of $D$ is defined as
$$
\text{vol}(D)=\text{vol}_{\mathbb{F}_\delta}(D):=\limsup_{m\to\infty}\dfrac{h^0(\mathbb{F}_\delta,mD)}{m^2/2},
$$
and $D$ is a big divisor if and only if $\text{vol}(D)>0$. By \cite[Corollary 2.3.22]{Laz1} it holds  that $\text{vol}(D)=P_D^2$. 

\begin{definition}\label{def_mupicoD}
Let $\nu_n$ be a divisorial valuation of $\mathbb{F}_\delta$ and $D$ a big divisor on $\mathbb{F}_\delta$. Fo\-llowing \cite{BouKurMacSze} and \cite{DumHarKurRoeSze}, we define the values $\mu_D(\nu_n)$ and $\hat{\mu}_D(\nu_n)$ as 
$$
\mu_D(\nu_n):=\max\{\nu_n(\varphi_{D'})\ | \ D'\in |D|\} \text{ and } \, \hat{\mu}_D(\nu_n):=\lim_{m\to\infty}\dfrac{\mu_{mD}(\nu_n)}{m},
$$
where $\varphi_{D'}$ is the germ of $D'$ at $p$.
\end{definition}
By Proposition 2.9 in \cite{BouKurMacSze}, it holds that
\begin{equation}\label{Condition_mupico}
\hat{\mu}_D(\nu_n)\geq \sqrt{\dfrac{\text{vol}(D)}{\text{vol}(\nu_n)}}.
\end{equation}

\begin{definition}\label{Def_minimal}
Let $\nu_n$ be a divisorial valuation of $\mathbb{F}_\delta$ and $D$ a big divisor on $\mathbb{F}_\delta$. The valuation $\nu_n$ is \emph{minimal with respect to $D$} if  $\hat{\mu}_D(\nu_n)= \sqrt{\text{vol}(D)/\text{vol}(\nu_n)}$.
\end{definition}

\begin{remark}\label{Remark_muequalmupico}
{\rm 
Let $\nu_n$ be a divisorial valuation of $\mathbb{F}_\delta$ and $Z$ the surface that it defines. Assume that $D$ is a big divisor on $\mathbb{F}_\delta$. Then, by Theorem 6.4 of \cite{LazMus}, it holds the equality
$$ 
\hat{\mu}_D(\nu_n)=\sup\{t\in\mathbb{Q}_+ \ | \ D^*-tE_n \text{ is big on $Z$}\},
$$ 
where $\mathbb{Q}_+$ is the set of non-negative rational numbers.
}
\end{remark}

Our next definition divides divisorial valuations $\nu_n$ of $\mathbb{F}_\delta$ in two types accor\-ding to the value $\delta$ and the point $p$ where $\nu_n$ is centered. This classification was introduced in \cite{GalMonMor}. 

\begin{definition}
\label{Def_spe_nspe_val}
Let $\nu_n$ be a valuation of the quotient field of $\mathcal{O}_{\mathbb{F}_\delta,p}$ centered at $\mathcal{O}_{\mathbb{F}_\delta,p}.$ The valuation $\nu_n$ is called to be {\it special} (with respect to $\mathbb{F}_\delta$ and $p$) when one of the following conditions holds:
\begin{enumerate}
\item $\delta = 0$.
\item $\delta >0$ and $p$ is a special point.
\item $\delta >0$, $p$ is a general point and there is no integral curve in the complete linear system $|M|,$ given by the section $M$, whose strict transform on $Z$ has negative self-intersection.
\end{enumerate}
The remaining valuations $\nu_n$ will be called {\it non-special}.
\end{definition}
Let $\nu_n$ be a divisorial valuation of $\mathbb{F}_\delta$. We denote by $F_1$ the fiber which goes through the point $p$ and, when $\nu_n$ is non-special, by  $M_1$ the unique integral curve in $|M|$ whose strict transform on $Z$  has negative self-intersection.

\medskip

Next we introduce the so-called non-positive at infinity valuations of $\mathbb{F}_\delta$. For valuations in this  family, we will be able to compute the value $\hat{\mu}_D(\nu_n)$ for any big divisor $D$.

\begin{definition}\label{Def_NPI_spe_nspe_val}
Let $\nu_n$ be a special (respectively, non-special) divisorial valuation of $\mathbb{F}_\delta$. The valuation $\nu_n$ is called {\it non-positive at infinity} whenever $\nu_n(h)\leq 0$ for all $h\in\mathcal{O}_{\mathbb{F}_\delta}(\mathbb{F}_\delta\setminus (F_1\cup M_0))$ (respectively, $h\in\mathcal{O}_{\mathbb{F}_\delta}(\mathbb{F}_\delta\setminus (F_1\cup M_1))$). 
\end{definition}

As a consequence of \cite[Theorem 3.6]{GalMonMor}, it is sufficient to check the condition $2\nu_n(\varphi_{M_0})\nu_n(\varphi_{F_1})+\delta\nu_n(\varphi_{F_1})^2\geq \text{vol}^{-1}(\nu_n)$ (respectively, $2\nu_n(\varphi_{M_1})\nu_n(\varphi_{F_1})$ $-\delta\nu_n(\varphi_{F_1})^2$ $\geq \text{vol}^{-1}(\nu_n)$) to decide whether a special (respectively, non-special) divisorial valuation $\nu_n$ of $\mathbb{F}_\delta$ is non-positive at infinite. Moreover, under this assumption,   the cone of curves of the surface $Z$ defined by $\nu_n$ is generated by the classes of the strict transforms of the divisors $F_1,M_0,E_1,\ldots, E_n$ (respectively, $F_1,M_0,M_1,E_1,\ldots, E_n$).

\medskip

To conclude this section, we determine the mentioned Seshadri-type constant for any non-positive at infinity divisorial valuation and big divisor of a Hirzebruch surface $\mathbb{F}_\delta$. We also extract some consequences of this result.

\begin{theorem}\label{Thm_mupico_nonpositive}
Let $\nu_n$ be a non-positive at infinity divisorial valuation of the quotient field of $\mathcal{O}_{\mathbb{F}_\delta,p}$ centered at $\mathcal{O}_{\mathbb{F}_\delta,p}$ and $D\sim aF+bM$ a big divisor on $\mathbb{F}_\delta$. Then,  
\begin{itemize}
\item[(a)] If $\nu_n$ is special then it holds that $\hat{\mu}_D(\nu_n)=(a+b\delta)\nu_n(\varphi_{F_1})+b\nu_n(\varphi_{M_0}).$
\item[(b)] Otherwise, $
\hat{\mu}_D(\nu_n)=a\nu_n(\varphi_{F_1})+b\nu_n(\varphi_{M_1}).$
\end{itemize}
\end{theorem}
\begin{proof}
For proving Statement (a) we assume that $p$ is a special point. When $p$ is a point of $\mathbb{F}_0$ (respectively, $p$ is general point), the proof is analogous by setting $\delta=0$ (respectively, $\nu_n(\varphi_{M_0})=0$). Let $C$ be a curve on $\mathbb{F}_\delta$ such that $C\in |mD|,$ where $m$ is a  positive integer, and denote by $\tilde{C}$ its strict transform on $Z$. By \cite[Theorem 3.6]{GalMonMor}, it holds that $\Lambda_n=\nu_n(\varphi_{M_0})F^*+\nu_n(\varphi_{F_1})M^*-\sum_{i=1}^n\nu_n(\mathfrak{m}_i)E_i^*$ is a nef divisor and then  $\Lambda_n\cdot \tilde{C}\geq 0$. This means that
$$
(a+b\delta)\nu_n(\varphi_{F_1})+b\nu_n(\varphi_{M_0})\geq \dfrac{\nu_n(\varphi_C)}{m}
$$
and, so,  we have found an upper bound for $\nu_n(\varphi_C)/m,$ where $C\in |mD|$ and $m$ a positive integer. Now, consider  the curve $C_1=m(a+\delta b)F_1 + mbM_0$, then  
$$
C_1\in |mD| \text{ and }\dfrac{\nu_n(\varphi_{C_1})}{m}=(a+\delta b)\nu_n(\varphi_{F_1})+b\nu_n(\varphi_{M_0}),
$$
which proves that the bound can be reached and Statement (a) holds. 

The proof of Statement (b) follows analogously by taking  the  divisor 
$$
\Delta_n=(\nu_n(\varphi_{M_1})-\delta\nu_n(\varphi_{F_1}))F^*+\nu_n(\varphi_{F_1})M^*-\sum_{i=1}^n\nu_n(\mathfrak{m}_i)E_i^*,
$$ 
which is nef by \cite[Theorem 4.8]{GalMonMor}, and the curve $C_1=maF_1+mbM_1$.
\end{proof}

\begin{corollary}\label{Cor_minimalcase}
Let $\nu_n$ be a non-positive at infinity divisorial valuation of $\mathbb{F}_\delta$ and $D\sim aF+bM$ a big and nef divisor on $\mathbb{F}_\delta$. Then, 
\begin{itemize}
\item[(a)] When $\nu_n$ is special, it is minimal with respect to $D$ if and only if $$2\nu_n(\varphi_{M_0})\nu_n(\varphi_{F_1})+\delta\nu_n(\varphi_{F_1})^2= \text{\emph{vol}}(\nu_n)^{-1}$$ and $a=b\nu_n(\varphi_{M_0})/\nu_n(\varphi_{F_1}).$ 
\item[(b)]Otherwise, $\nu_n$ is minimal with respect to $D$ if and only if $$2\nu_n(\varphi_{M_1})\nu_n(\varphi_{F_1})-\delta\nu_n(\varphi_{F_1})^2= \text{\emph{vol}}(\nu_n)^{-1}$$ and $a=b(\nu_n(\varphi_{M_1})-\delta\nu_n(\varphi_{F_1}))/\nu_n(\varphi_{F_1}).$
\end{itemize}  
\end{corollary}
\begin{proof}
We will prove Item (a) in the case when $p_1$ is a special point; when $p_1\in\mathbb{F}_0$ (respectively, $p_1$ is a general point) the proof is analogous and follows taking $\delta=0$ (respectively, $\nu_n(\varphi_{M_0})=0$). A proof for Item (b) also runs similarly.

For a start, we are going to prove the minimality of $\nu$ under the conditions of the statement. Taking into account that $2\nu_n(\varphi_{M_0})\nu_n(\varphi_{F_1})+\delta\nu_n(\varphi_{F_1})^2= \text{vol}(\nu_n)^{-1}$, one obtains that 
\begin{align*}
\dfrac{\text{vol}(D)}{\text{vol}(\nu_n)}=&(2ab+b^2\delta)\delta\nu_n(\varphi_{F_1})^2+ 2b(a+b\delta)\nu_n(\varphi_{F_1})\nu_n(\varphi_{M_0}) + 2ab\nu_n(\varphi_{F_1})\nu_n(\varphi_{M_0})\\
=&(a+b\delta)^2\nu_n(\varphi_{F_1})^2+2b(a+b\delta)\nu_n(\varphi_{F_1})\nu_n(\varphi_{M_0}) + b^2\nu_n(\varphi_{M_0})^2\\
=&\hat{\mu}_{D}(\nu_n)^2,
\end{align*}
where the second equality holds since $(a\nu_n(\varphi_{F_1})-b\nu_n(\varphi_{M_0}))^2=0$, which proves that $\nu_n$ is minimal with respect to $D$.

Now assume that $\nu_n$ is minimal with respect to $D$. Then, by Theorem \ref{Thm_mupico_nonpositive}, it holds that
\begin{equation}\label{Cond_minimalconcoeff_a_b}
\left((a+b\delta)\nu_n(\varphi_{F_1})+b\nu_n(\varphi_{M_0})\right)^2=b(2a+\delta b)\text{vol}(\nu_n)^{-1}.
\end{equation}
On the other hand, one has the equality 
\begin{align*}
\left((a+b\delta)\nu_n(\varphi_{F_1})+b\nu_n(\varphi_{M_0})\right)^2=& \ (a\nu_n(\varphi_{F_1})-b\nu_n(\varphi_{M_0}))^2\\ & + b(2a+\delta b)(2\nu_n(\varphi_{F_1})\nu_n(\varphi_{M_0})+\delta\nu_n(\varphi_{F_1})^2),
\end{align*}
which together with Equality  \eqref{Cond_minimalconcoeff_a_b} gives rise to  
$$
(a\nu_n(\varphi_{F_1})-b\nu_n(\varphi_{M_0}))^2 + b(2a+\delta b)(2\nu_n(\varphi_{F_1})\nu_n(\varphi_{M_0})+\delta\nu_n(\varphi_{F_1})^2 - \text{vol}(\nu_n)^{-1})=0.
$$
Both addends of the above expression are not negative, so they must vanish. This completes the proof.

\end{proof}

\begin{corollary}\label{Cor_nonminimalresptanyD}
Let $\nu_n$ be a non-positive at infinity divisorial valuation of $\mathbb{F}_\delta$. Then, $\nu_n$ is non-minimal with respect to any big divisor $D$ on $\mathbb{F}_\delta$ whenever some of the following conditions holds:
\begin{itemize}
\item[(a)] $\nu_n$ is special and $2\nu_n(\varphi_{M_0})\nu_n(\varphi_{F_1})+\delta\nu_n(\varphi_{F_1})^2> \text{\emph{vol}}(\nu_n)^{-1}$. 
\item[(b)] $\nu_n$ is non-special and  $2\nu_n(\varphi_{M_1})\nu_n(\varphi_{F_1})-\delta\nu_n(\varphi_{F_1})^2> \text{\emph{vol}}(\nu_n)^{-1}$. 
\end{itemize}
\end{corollary}

\begin{proof}
We begin by proving Item (a). We only need to show that 
$$
\hat{\mu}_D(\nu_n)^2/P_D^2>\overline{\beta}_{g+1}(\nu_n)
$$ 
holds for any big divisor $D\sim aF+bM$, $P_D$ being its positive part in the Zariski decomposition. 
Firstly, assume that $\delta>0$. Let $q:(-\delta,\infty)\cap\mathbb{Q}\to\mathbb{Q}_+$ be the map 
$$
q(x):=\left\lbrace\begin{array}{cc}
\dfrac{((x + \delta)\nu_n(\varphi_{F_1})+\nu_n(\varphi_{M_0}))^2}{((1/\delta)x+1)^2\delta}  \text{ if }x\in(-\delta,0)\cap\mathbb{Q}, \\[3mm]
\dfrac{((x + \delta)\nu_n(\varphi_{F_1})+\nu_n(\varphi_{M_0}))^2}{2x + \delta}  \text{ if }x\in[0,\infty)\cap\mathbb{Q}. \\
\end{array}\right.
$$ 
Notice that $q$ has an absolute minimum at the point $(x_1,q(x_1)),$ where $x_1=\nu_n(\varphi_{M_0})/\nu_n(\varphi_{F_1})$ and 
$$
q(x_1)=2\nu_n(\varphi_{M_0})\nu_n(\varphi_{F_1})+\nu_n(\varphi_{F_1})^2\delta.
$$ 
Since $q(a/b)=\hat{\mu}_D(\nu_n)^2/P_D^2$ (by Theorem \ref{Thm_mupico_nonpositive}) we have that
$$\hat{\mu}_D(\nu_n)^2/P_D^2\geq 2\nu_n(\varphi_{M_0})\nu_n(\varphi_{F_1})+\nu_n(\varphi_{F_1})^2\delta>\text{vol}(\nu_n)^{-1}=\overline{\beta}_{g+1}(\nu_n).$$

If $\delta=0$, by Theorem \ref{Thm_mupico_nonpositive} it holds that
$$\hat{\mu}_D(\nu_n)^2/P_D^2-2\nu_n(\varphi_{M_0})\nu_n(\varphi_{F_1})=(a\nu_n(\varphi_{F_1})-b\nu_n(\varphi_{M_0}))^2/P_{D}^2\geq 0.$$
Hence $\hat{\mu}_D(\nu_n)^2/P_D^2\geq 2\nu_n(\varphi_{M_0})\nu_n(\varphi_{F_1})>\text{vol}(\nu_n)^{-1}=\overline{\beta}_{g+1}(\nu_n).$

\medskip

To conclude, we notice that Item (b) can be proved following the same reasoning of the proof of Item (a), but considering the map $q_1:(-\delta,\infty)\cap\mathbb{Q}\to\mathbb{Q}_+,$
$$
q_1(x):=\left\lbrace\begin{array}{cc}
\dfrac{(\nu_n(\varphi_{F_1})x+\nu_n(\varphi_{M_1}))^2}{((1/\delta)x+1)^2\delta}  \text{ if }x\in(-\delta,0)\cap\mathbb{Q}, \\[3mm]
\dfrac{(\nu_n(\varphi_{F_1})x+\nu_n(\varphi_{M_1}))^2}{2x + \delta}  \text{ if }x\in[0,\infty)\cap\mathbb{Q}, \\
\end{array}\right.
$$
instead of $q$.
\end{proof}

\section{Newton-Okounkov bodies of non-positive at infinity valuations}\label{Sec_NOBFdelta}
Let $X$ be a smooth complex projective surface. A sequence 
$$
C_\bullet:=\{X \supset C \supset \{q\} \},
$$ 
where $C$ is an smooth irreducible curve on $X$ and $q$ a closed point of $C,$ is called a \emph{flag} of $X$. The point $q$ is the \emph{center} of $C_\bullet$. 

In this section we study the Newton-Okounkov bodies with respect to a flag 
\begin{equation}\label{Equ_FlagExceptionaldivisor}
E_\bullet:=\{Z=Z_r \supset E_r \supset \{p_{r+1}\}\},
\end{equation}
where $Z=Z_r$ is the surface defined by a finite simple sequence of blowing-ups as in \eqref{Eq_sequencepointblowingups} with $Z_0=\mathbb{F}_\delta$ and $E_r$ the last exceptional divisor created. We denote by $p_{r+1}$ the center of $E_\bullet$.

Flags of smooth varieties (not only surfaces) define and are defined by discrete valuations whose rank coincides with the dimension of the variety. In our case, they correspond  one-to-one with exceptional curve valuations $\nu$ (up to equivalence) whose configuration of infinitely near points $\mathcal{C}_\nu=\{p_i\}_{i=1}^\infty$ satisfies that the points  $\{p_i\}_{i=1}^r$ are given by the divisorial valuation $\nu_r$ defined by $E_r$ and the remaining points $p_i$, for $i>r$, are proximate to  $p_r$. If the point $p_{r+1}$ is satellite then there exists an exceptional divisor $E_\eta$ such that $\eta\neq r$ and $p_{r+1}\in E_\eta$. 

According to \cite[Section 3.2]{GalMonMoyNic2}, the flag valuation $\nu:=\nu_{E_\bullet},$ defined by $E_\bullet$,  satisfies that, for $f\in R=\mathcal{O}_{\mathbb{F}_\delta,p}$,  $\nu_{E_\bullet}(f)=(\upsilon_1(f),\upsilon_2(f))$ with $\upsilon_1(f)=\nu_r(f)$ and $
\upsilon_2(f):=\nu_\eta(f) + \sum_{p_i\to p_r} \text{mult}_{p_i}(f),$ where $\nu_\eta$ is the divisorial valuation defined by $E_\eta$. Up to equivalence of valuations, the value group of $\nu$ is $\mathbb{Z}^2$ and $\nu(\mathfrak{m}_{r})=(1,0)$ and $\nu(\mathfrak{m}_{r+1})=(0,1)$.

\begin{definition}\label{Def_spe_nonspe_except_val}
Let $\nu$ be a exceptional curve valuation of $\mathbb{F}_\delta$ and $D$ a big divisor on $\mathbb{F}_\delta$. The valuation $\nu$ is \emph{minimal with respect to $D$} whenever its first component $\nu_r$ is minimal with respect to $D$. The valuation $\nu$ is called \emph{special} (respectively, \emph{non-special}) when its first component $\nu_r$ is a special (respectively, non-special) divisorial valuation of $\mathbb{F}_\delta$. Analogously, $\nu$ is \emph{non-positive at infinity} whenever $\nu_r$ is non-positive at infinity.
\end{definition}

Newton-Okounkov bodies are non-empty convex and compact objects attached to flags and give very interesting geometric information \cite{LazMus,KavKho,BouKurMacSze}. The goal of this section is to explicitly compute the Newton-Okounkov bodies $\Delta_{\nu_{E_\bullet}}(D^*)$, where $E_\bullet$ is a flag as in \eqref{Equ_FlagExceptionaldivisor} corresponding to a non-positive at infinity exceptional curve valuation $\nu_{E_\bullet}$ and $D^*$ is the pull-back on $Z$ of a big divisor $D$ on $\mathbb{F}_\delta$. We start by defining the Newton-Okounkov body in our case. 
\begin{definition}\label{def_NObody}
Let $\nu$ be an exceptional curve valuation of $\mathbb{F}_\delta$ and $D$ a big divisor on $\mathbb{F}_\delta$. The \emph{Newton-Okounkov body of $D$ with respect to $\nu$} is defined as 
$$
\Delta_{\nu}(D):=\overline{\bigcup_{m\geq 1}\bigg\{\dfrac{\nu (f)}{m}\ | \ f\in H^0(\mathbb{F}_\delta,mD)\setminus \{0\}\bigg\}},
$$ 
where the upper line means the closed convex hull in $\mathbb{R}^2$.
\end{definition}
Notice that, if $E_\bullet$ is a flag as in (\ref{Equ_FlagExceptionaldivisor}) and $\nu=\nu_{E_\bullet}$, then $\Delta_{\nu}(D)=\Delta_{\nu_{E_\bullet}}(D^*)$. Moreover, the Newton-Okounkov body is  a  polygon (see \cite{KurLozMac}) and 
$$
\text{vol}(D)=\text{vol}_Z(D^*)=2\,\text{vol}_{\mathbb{R}^2}(\Delta_{\nu}(D)),
$$
where $\text{vol}_{\mathbb{R}^2}$ means Euclidean area (see \cite{LazMus}). 

Set $g+1$ (respectively, $g^*+1$) the number of generators of the semigroup of the divisorial valuation $\nu_r$ (respectively, exceptional curve valuation $\nu$). It holds that $g^*=g$ when $p_r$ and $p_{r+1}$ are satellite points. Otherwise, $g^*=g+1$. Denote by $S_\nu$ the semigroup of values of $\nu$, that is, the monoid
$$
S_\nu:=\{\nu(f)\ | \ f\in R\setminus \{0\}\}\subseteq \mathbb{Z}^2,
$$
endowed with the lexicographical ordering.  As mentioned, the set $S_\nu$ is ge\-ne\-ra\-ted by the set of pairs $\{\overline{\beta}_i(\nu)\}_{i=0}^{g^*}$ (respectively, $\{\overline{\beta}_i(\nu)\}_{i=0}^{g+1}$), where $\overline{\beta}_i(\nu)=(\overline{\beta}_i(\nu_r),\overline{\beta}_i(\nu_\eta))$ (respectively, $\overline{\beta}_i(\nu)=(\overline{\beta}_i(\nu_r),0)$ and $\overline{\beta}_{g+1}(\nu)=(\overline{\beta}_{g+1}(\nu_r),1)$), whenever $p_{r+1}$ is a satellite (respectively, free) point.  

Let $\mathfrak{C}(\nu)$ be the convex cone of $\mathbb{R}^2$ spanned by $S_\nu$ and $\mathfrak{H}_D(\nu)$ the half-plane 
$\{(x,y)\in\mathbb{R}^2 \ | \ x\leq \hat{\mu}_D(\nu_r)\}$. 
Then, the next result follows from Definitions \ref{def_NObody} and \ref{def_mupicoD} and \cite[Proposition 3.6]{GalMonMoyNic2}.

\begin{proposition}\label{Prop_coortriangle}
The set $\mathfrak{C}(\nu)\cap\mathfrak{H}_D(\nu)$ is a triangle, which contains the Newton-Okounkov body $\Delta_\nu(D),$ whose vertices are 
$$
(0,0),\ \ \left(\hat{\mu}_D(\nu_r),\dfrac{\hat{\mu}_D(\nu_r)\overline{\beta}_0(\nu_\eta)}{\overline{\beta}_0(\nu_r)}\right) \ \ \text{ and } \ \ \left(\hat{\mu}_D(\nu_r),\dfrac{\hat{\mu}_D(\nu_r)\overline{\beta}_{g^*}(\nu_\eta)}{\overline{\beta}_{g^*}(\nu_r)}\right)
$$
whenever $q=p_{r+1}$ is the satellite point $E_r\cap E_\eta$ (with $\eta\neq r$); and 
$$
(0,0),\ \ \left(\hat{\mu}_D(\nu_r),0\right) \ \ \text{ and } \ \ \left(\hat{\mu}_D(\nu_r),\dfrac{\hat{\mu}_D(\nu_r)}{\overline{\beta}_{g+1}(\nu_r)}\right),
$$
otherwise.
\end{proposition}

Our next result determines the Newton-Okounkov bodies of minimal exceptional curve valuations of Hirzebruch surfaces.

\begin{theorem}\label{Thm_NOBminimalvaluation}
Let $\nu$ be an exceptional curve valuation of a Hirzebruch surface $\mathbb{F}_\delta$ and $D$ a big divisor on $\mathbb{F}_\delta$. Then, the Newton-Okounkov body $\Delta_\nu(D)$ coincides with the triangle $\mathfrak{C}(\nu)\cap \mathfrak{H}_D(\nu)$ if and only if $\nu$ is minimal with  respect to $D$.
\end{theorem}
\begin{proof}
Proposition \ref{Prop_coortriangle} and \cite[Lemma 3.9]{GalMonMoyNic2} prove that the Newton-Okounkov body $\Delta_\nu(D)$ is contained in the triangle $\mathfrak{C}(\nu)\cap \mathfrak{H}_D(\nu)$ whose area is 
$$
\hat{\mu}_D(\nu_r)^2/2\overline{\beta}_{g+1}(\nu_r).$$
Taking into account that  
$$
\dfrac{\hat{\mu}_D(\nu_r)^2}{2\overline{\beta}_{g+1}(\nu_r)}\geq\dfrac{\text{vol}(D)}{2}
$$
by inequality \eqref{Condition_mupico}, it holds that the triangle $\mathfrak{C}(\nu)\cap \mathfrak{H}_D(\nu)$ will coincide with the Newton-Okounkov body $\Delta_\nu(D)$ when both figures have the same area, which is true only when $\nu$ is minimal with respect to $D$.   
\end{proof}

From now on suppose that $\nu$ is an exceptional curve valuation of $\mathbb{F}_\delta$ which is non-minimal with  respect to a big divisor $D\sim aF+bM$.

When $\delta=0,$ $D$ is also nef. Otherwise ($\delta\neq 0$), $D^*$ may be big but not nef. In this last case the positive and negative parts of the Zariski decomposition of $D^*$ are
$$P_{D^*}\sim \left(b+\dfrac{a}{\delta}\right)M^*\text{ and }N_{D^*}=\dfrac{-a}{\delta}\tilde{M}_0+\sum_{i=0}^{i_{M_0}}\dfrac{-a\nu_i(\varphi_{M_0})}{\delta} E_i,$$
where $i_{M_0}$ indicates the last point in $C_\nu$ through which the strict transform of $M_0$ passes. Then we distinguish two cases.

Case 1. $p_{r+1}$ belongs to the support of $N_{D^*},$ $\text{supp}(N_{D^*}),$ of the negative part of the Zariski decomposition of the divisor $D^*$. This fact holds if and only if $g^*=1,p_1$ is a special point, all the points in $\{p_i\}_{i=1}^{r+1}$ are free, $i_{M_0}=r+1$ and $D$ is big and not nef. %\edz{Case 1? CJ:Sí, están todas las condiciones. Aunque el iff se tiene por la desc. Zar. de $D^*$ (está escrita en el caso 2).}

Case 2. $p_{r+1}\not\in\text{supp}(N_{D^*}).$ In this case, to compute $\Delta_\nu(D)$ we can assume that $D$ is nef because  this assumption does not produce loss of generality.  Indeed, if the divisor $D$ is big but not nef, then $b> 0$ and $-b\delta < a < 0$, and, as  $p_{r+1}\not\in\text{supp} (N_{D^*})$, by \cite[Lemma 1.10]{KurLoz1}, it holds that 
$$
\Delta_{\nu}(D)=\Delta_{\nu}(P_D)=\left(b+\dfrac{a}{\delta}\right)\Delta_{\nu}(M).
$$

Notice that, $\text{vol}(D)=D^2$ and Inequality  \eqref{Condition_mupico} can be written as 
\begin{equation}\label{Cond_valoracionminimalD2beta}
\hat{\mu}_D(\nu_r)\geq \sqrt{D^2\overline{\beta}_{g+1}(\nu_r)}
\end{equation}
if $p_{r+1}\not\in\text{supp} (N_{D^*})$. Otherwise, we will replace $D$ by $P_D$.
\medskip

In the following subsections we will explicitly get the Newton-Okounkov body of divisors $D$ with respect to non-positive at infinity $\nu$. We start with special va\-lua\-tions, where the case $p_{r+1}\in\text{supp}(N_{D^*})$ might happen.

\subsection{Newton-Okounkov bodies with respect to non-positive at infi\-nity special valuations}\label{Subsec_NOofNPIvalspec}

Along this section $D\sim aF+bM$ will be a big divisor on $\mathbb{F}_\delta$ and $\nu$ a non-positive at infinity special exceptional curve valuation of $\mathbb{F}_\delta$ whose first component is $\nu_r$. Recall that $\nu$ is not minimal with respect to $D$.

The symbol $\theta_1^r(D)$ will stand for $a\nu_r(\varphi_{F_1})-b\nu_r(\varphi_{M_0}),$ where $F_1$ is the fiber which passes through $p$ and $M_0$ the special section. When $\theta_1^r(D)=0,$
it holds that either $a=b\nu_r(\varphi_{M_0})/\nu_r(\varphi_{F_1})$, or $\nu_r(\varphi_{M_0})=0$ and $a=0$. Notice that, in the second case, some objects that we will introduce are not defined and we will avoid using them. Moreover, if $p_{r+1}\in \text{supp}(N_{D^*}),$ then $\theta_1^r(D)$ is always negative.

We start by stating two lemmas which allow us to compute the Zariski decomposition of some key divisors.

\begin{lemma}\label{Lemma_partepositivanef_casoespecial}
Let $\nu_r$ be a non-positive at infinity special  divisorial valuation of $\mathbb{F}_\delta$ and $D$ a big and nef divisor on $\mathbb{F}_\delta$. Let also  $\theta_1^r(D)$ be as in the above paragraphs. Then, the divisor 
$$
D_1=D^*-\dfrac{b}{\nu_r(\varphi_{F_1})}\sum_{i=1}^r\nu_r(\mathfrak{m}_i)E_i^* \left(\!\!\text{respectively, } D_2=D^*-\dfrac{a}{\nu_r(\varphi_{M_0})}\sum_{i=1}^r\nu_r(\mathfrak{m}_i)E_i^*\right)
$$
is nef when $\theta_1^r (D)\geq 0$ (respectively, $\theta_1^r (D)<0$). 
\end{lemma}
\begin{proof}
We will prove that $D_1$ is nef. A proof for $D_2$ runs similarly. As $b$ is a positive integer, one can obtain that   
\begin{align*}
D_1 & = D^*-\dfrac{b}{\nu_r(\varphi_{F_1})}\sum_{i=1}^r\nu_r(\mathfrak{m}_i)E_i^*\\
& \sim \dfrac{b}{\nu_r(\varphi_{F_1})}\left( \dfrac{a\nu_r(\varphi_{F_1})}{b}F^*+\nu_r(\varphi_{F_1})M^*- \sum_{i=1}^r\nu(\mathfrak{m}_i)E_i^*\right)\\
& = \dfrac{b}{\nu_r(\varphi_{F_1})}\left(\dfrac{\theta_1^r(D)}{b}F^* + \Lambda_r\right),
\end{align*}
where $\Lambda_r=\nu_r(\varphi_{M_0})F^*+\nu_r(\varphi_{F_1})M^* - \sum_{i=1}^r\nu_r(\mathfrak{m}_i)E_i^*$. The divisors $F^*$ and $\Lambda_r$ are nef by \cite[Theorem 3.6]{GalMonMor} and then $D_1$ is also a nef divisor since $\theta_1^r(D)$ is non-negative.
\end{proof}

\begin{lemma}\label{Lemma_t_ibelongstoT_casoespecial}
Let $\nu_r$ be a non-positive at infinity special  divisorial valuation of $\mathbb{F}_\delta$ and $Z$ the surface that it defines. Consider a big and nef divisor $D\sim aF+bM$ and write 
$\theta_1^r (D):=a\nu_r(\varphi_{F_1})-b\nu_r(\varphi_{M_0})$. Then, the following four rational numbers:   
\[
t_1=\dfrac{b}{\nu_r(\varphi_{F_1})}\overline{\beta}_{g+1}(\nu_r),\ \ t_2=\dfrac{b}{\nu_r(\varphi_{F_1})}\overline{\beta}_{g+1}(\nu_r)+\theta_1^r (D), \ \ 
\]
\[
t_3=\dfrac{a}{\nu_r(\varphi_{M_0})}
\overline{\beta}_{g+1}(\nu_r)\text{ and } 
t_4=\dfrac{(a+b\delta)\overline{\beta}_{g+1}(\nu_r) - \theta_1^r (D)\nu_r(\varphi_{M_0})}{\nu_r(\varphi_{M_0})+\delta \nu_r(\varphi_{F_1})}
\]
satisfy that $t_1$ and $t_2$ (respectively, $t_3$ and $t_4$) belong to the set 
$$
T_{D,\nu_r}:=\{t \in\mathbb{Q}\ | \ 0\leq t\leq \hat{\mu}_D(\nu_r)\}
$$  
when $\theta^r_1(D)\geq 0$ (respectively, $\theta^r_1(D)<0$). In addition, if $p_{r+1}\in\text{\emph{supp}}(N_{D^*}),$ then $-a\nu_r(\varphi_{M_0})/\delta < t_4\leq \hat{\mu}_{D}(\nu_r)$.
\end{lemma}
\begin{proof}
We will show that $t_1,t_2\leq \hat{\mu}_D(\nu_r)$. A proof for the other cases runs similarly.

Let us prove that $t_1\leq \hat{\mu}_D(\nu_r)$ when $\theta_1^r(D)\geq 0$. By Lemma \ref{Lemma_partepositivanef_casoespecial}, it holds that the divisor $D_1=D^*-\frac{b}{\nu_r(\varphi_{F_1})}\sum_{i=1}^r\nu_r(\mathfrak{m}_i)E_i^*$ is nef and then, for any curve $C\in|mD|$, $m$ being a positive integer, one has that  
$$
 m(2ab+b^2\delta)-\dfrac{b}{\nu_r(\varphi_{F_1})}\nu_r(\varphi_C)=D_1\cdot \tilde{C}\geq 0,
$$
where $\tilde{C}$ is the strict transform of $C$ under the birational map defined by $\nu_r$. This shows that 
$$
2ab+b^2\delta\geq \dfrac{b}{\nu_r(\varphi_{F_1})}\hat{\mu}_D(\nu_r),
$$
which together with Inequality  \eqref{Cond_valoracionminimalD2beta} allow us to deduce the inequalities 
$$
\hat{\mu}_D(\nu_r)\geq \dfrac{D^2\overline{\beta}_{g+1}(\nu_r)}{\hat{\mu}_D(\nu_r)}=\dfrac{(2ab+b^2\delta)\overline{\beta}_{g+1}(\nu_r)}{\hat{\mu}_D(\nu_r)}\geq \dfrac{b\overline{\beta}_{g+1}(\nu_r)}{\nu_r(\varphi_{F_1})},
$$
which prove our statement.

To finish the proof, we will see that $t_2\leq \hat{\mu}_D(\nu_r)$ when $\theta_1^r(D)\geq 0$. By Theorem  \ref{Thm_mupico_nonpositive}, it suffices to prove the inequality 
$$
b(2\nu_r(\varphi_{M_0})\nu_r(\varphi_{F_1})+\delta\nu_r(\varphi_{F_1})^2)\geq b\overline{\beta}_{g+1}(\nu_r),
$$
which holds by \cite[Theorem 3.6]{GalMonMor} after noticing that $b$ is a positive integer.
\end{proof}

\begin{remark}
Theorem \ref{Thm_mupico_nonpositive} and Corollary \ref{Cor_minimalcase} prove that  
$$
\hat{\mu}_D(\nu_r)=b\overline{\beta}_{g+1}(\nu_r)/\nu_r(\varphi_{F_1})=t_1=t_2\,(=t_3=t_4,\text{ when }\nu_r(\varphi_{M_0})\neq 0)
$$
whenever the valuation $\nu_r$ is minimal with respect to $D$.
 
Otherwise, Lemma  \ref{Lemma_t_ibelongstoT_casoespecial} provides two values, $t_1$ and $t_2$ (respectively, $t_3$ and $t_4$) when $\theta_1^r(D)\geq 0$ (respectively, $\theta_1^r(D)<0$)  and $D$ is big and nef. If $\theta_1^r(D)=0$, then $\hat{\mu}_D(\nu_r)>t_1=t_2\,(=t_3=t_4,\text{ when }\nu_r(\varphi_{M_0})\neq 0)$, and when $\delta>0, a=0$ and $\theta_1^r(D)<0,$ then $t_3=0$. Moreover, if $2\nu_r(\varphi_{M_0})\nu_r(\varphi_{F_1})+\delta\nu_r(\varphi_{F_1})^2 = \overline{\beta}_{g+1}(\nu_r)$ holds, we obtain that $t_2=\hat{\mu}_D(\nu_r)$ (respectively, $t_4=\hat{\mu}_D(\nu_r)$) whenever $\theta_1^r(D)>0$ (respectively, $\theta_1^r(D)<0$). Finally, if $p_{r+1}\in\text{supp}(N_{D^*})$, the value $t_4$ defined in Lemma \ref{Lemma_t_ibelongstoT_casoespecial} satisfies $t_4=\hat{\mu}_D(\nu_r)$ when $2\nu_r(\varphi_{M_0})\nu_r(\varphi_{F_1})+\delta\nu_r(\varphi_{F_1})^2 = \overline{\beta}_{g+1}(\nu_r).$
\end{remark}

\begin{lemma}\label{Lemma_Ntdeterminesmatrixnegativedefinite_casoespecial}
Let $\nu_r$ be a special divisorial valuation of $\mathbb{F}_\delta$ and $D$ a big and nef divisor. Suppose also that $\nu_r$ is non-minimal with respect to $D$. Let $\nu_i$ be  the divisorial valuation defined by the exceptional divisor $E_i$, $1\leq i \leq r-1$. Then, the intersection matrices determined by the families of divisors $\{\tilde{F}_1,E_1,\ldots,E_{r-1}\}$ and $\{\tilde{M}_0,E_1,\ldots,E_{r-1}\}$ are negative definite. In addition, when $p_{r+1}\in\text{\emph{supp}} (N_{D^*}),$ $\{\tilde{M}_0,E_1,\ldots,E_{r-1}\}$ also determines a negative definite intersection matrix. Notice that in this last case, $D$ is big and not nef.
\end{lemma}
\begin{proof}
Consider the divisor $D_1$ defined in Lemma \ref{Lemma_partepositivanef_casoespecial}. We showed that it is nef, let us see that it is also big. Indeed, 
$$
D_1^2=D^2 - \dfrac{b^2\overline{\beta}_{g+1}(\nu_r)}{\nu_r(\varphi_{F_1})^2}\geq D^2 -\dfrac{b\hat{\mu}_{D}(\nu_r)}{\nu_r(\varphi_{F_1})}\left(\dfrac{D^2\overline{\beta}_{g+1}(\nu_r)}{\hat{\mu}_D(\nu_r)^2}\right)> D^2-\dfrac{b\hat{\mu}_D(\nu_r)}{\nu_r(\varphi_{F_1})}\geq 0,
$$
where the second inequality holds since $\nu_r$ is non-minimal with respect to $D$ and the last one by the proof of Lemma \ref{Lemma_t_ibelongstoT_casoespecial}. So, $D_1$ is a big divisor by \cite[Theorem 2.2.16]{Laz1}. Finally, the facts that $D_1\cdot \tilde{F}_1=0$ and $D_1\cdot E_i=0$ for $1\leq i \leq r-1$ prove our statement for $\{\tilde{F}_1,E_1,\ldots,E_{r-1}\}$ by Lemma 4.3 of \cite{BauKurSze}. The remaining cases can be proved analogously either with the divisor $D_2$ in Lemma \ref{Lemma_partepositivanef_casoespecial} or with the nef and big divisor $(b+a/\delta)M^*$. 
\end{proof}

Our next result gives the positive part and the negative  part of the Zariski decomposition of certain divisors which will be useful.

\begin{proposition}\label{Prop_Zardecomp_casoespecial}
Let $\nu_r$ be a non-positive at infinity special divisorial valuation of $\mathbb{F}_\delta$, $Z$ the surface defined by $\nu_r$ and $\nu_i$ the divisorial valuation defined by the exceptional divisor $E_i$, $1\leq i \leq r-1$. Set $D\sim aF+bM$ a big and nef divisor on $\mathbb{F}_\delta$ and suppose that $\nu_r$ is non-minimal with respect to $D$. Write $\theta_1^r (D)=a\nu_r(\varphi_{F_1})-b\nu_r(\varphi_{M_0})$ and $\Lambda_r=\nu_r(\varphi_{M_0})F^*+\nu_r(\varphi_{F_1})M^*-\sum_{i=1}^r\nu_r(\mathfrak{m}_i)E_i^*$. Consider the divisors $D_1$ and $D_2$  in Lemma \ref{Lemma_partepositivanef_casoespecial} and the rational numbers $t_1,t_2,t_3$ and $t_4$ given in Lemma  \ref{Lemma_t_ibelongstoT_casoespecial}. Then, 

\begin{itemize}
\item[(a)]Assuming $\theta_1^r(D)\geq 0$, the positive and negative parts of the Zariski decomposition of the divisors $D_{t_1}:=D^* -t_1E_r$, and $D_{t_2}:=D^*-t_2E_r$ are 
\begin{equation*}
\begin{array}{c}
P_{D_{t_1}}\sim D_1 \ \text{ and } \  N_{D_{t_1}} = \dfrac{b}{\nu_r(\varphi_{F_1})}\displaystyle\sum_{i=1}^{r-1}\nu_r(\varphi_i)E_i, \\[4mm]
\text{and } 
P_{D_{t_2}}\sim \dfrac{b}{\nu_r(\varphi_{F_1})}\Lambda_r  \text{ and } \\[4mm]
N_{D_{t_2}}= \dfrac{\theta_1^r(D)}{\nu_r(\varphi_{F_1})}\tilde{F}_1+\displaystyle\sum_{i=1}^{r-1} \dfrac{b\nu_r(\varphi_i) + \theta_1^r(D)\nu_i(\varphi_{F_1}) }{\nu_r(\varphi_{F_1})}E_i.
\end{array}
\end{equation*}
\item[(b)] When $\theta_1^r(D)<0$, the positive and negative parts of the Zariski decomposition of $D_{t_3}:=D^* -t_3E_r$, and $D_{t_4}:=D^*-t_4E_r$ are 
\begin{equation*}
\begin{array}{c}
P_{D_{t_3}}\sim D_2 \ \text{ and } \  N_{D_{t_3}} = \dfrac{a}{\nu_r(\varphi_{M_0})}\displaystyle\sum_{i=1}^{r-1}\nu_r(\varphi_i)E_i,\\[4mm] 
\text{and } P_{D_{t_4}}\sim \dfrac{a+b\delta}{\nu_r(\varphi_{M_0})+\delta \nu_r(\varphi_{F_1})}\Lambda_r  \text{ and }\\[4mm] 
N_{D_{t_4}}= \left(\dfrac{-\theta_1^r(D)}{\nu_r(\varphi_{M_0})+\delta\nu_r(\varphi_{F_1})}\right)\tilde{M}_0\\[4mm]
+\displaystyle\sum_{i=1}^{r-1}\dfrac{(a+b\delta)\nu_r(\varphi_i)-\theta_1^r(D)\nu_i(\varphi_{M_0})}{\nu_r(\varphi_{M_0})+\delta\nu_r(\varphi_{F_1})}E_i,
\end{array}
\end{equation*}
\end{itemize}

Moreover, if $p_{r+1}\in\text{\emph{supp}} (N_{D^*}),$ then the positive and negative parts of $D_{t_4}$ are the divisors $P_{D_{t_4}}$ and $N_{D_{t_4}}$ described before.
\end{proposition}
\begin{proof}
We only prove Statement (a) since a similar proof can be given for  the remaining cases. Let us start with the decomposition of $D_{t_1}.$ It is clear that $P_{D_{t_1}}+N_{D_{t_1}}\sim D_{t_1}$. Also $P_{D_{t_1}}$ is nef, by  Lemma \ref{Lemma_partepositivanef_casoespecial}, and orthogonal  to each component of $N_{D_{t_1}}$, by the proximity equalities.  This concludes the proof after taking into account that the components of $N_{D_{t_1}}$ determine an intersection matrix which is negative definite.

Finally, we prove the claim for the divisor $D_{t_2}$. By \cite[Proposition 3.3 and Theorem 3.6]{GalMonMor},  $P_{D_{t_2}}$ is nef and  orthogonal to each component of $N_{D_{t_2}}$. As well, it follows from Lemma \ref{Lemma_Ntdeterminesmatrixnegativedefinite_casoespecial} that the intersection matrix determined by the components of $N_{D_{t_2}}$ is negative definite.  To conclude, summing the following two expressions: 
$$
D-\dfrac{b}{\nu_r(\varphi_{F_1})}\overline{\beta}_{g+1}(\nu_r)E_r\sim \dfrac{b}{\nu_r(\varphi_{F_1})}\Lambda_r + \dfrac{\theta_1^r(D)}{\nu_r(\varphi_{F_1})}F^*+\dfrac{b}{\nu_r(\varphi_{F_1})}\sum_{i=1}^{r-1}\nu_r(\varphi_i)E_i 
$$
and 
$$
-\theta_1^r(D)E_r = \dfrac{\theta_1^r(D)}{\nu_r(\varphi_{F_1})}\left(\sum_{i=1}^{r-1}\nu_i(\varphi_{F_1})E_i-\sum_{i=1}^{i_{F_1}}E_i^*\right),
$$
and taking into account that $\tilde{F}_1\sim F^*-\sum_{i=1}^{i_{F_1}}E_i^*,$ where $i_{F_1}$ indicates the last point in the configuration of infinitely near points $\mathcal{C}_{\nu_r}$ of the valuation $\nu_r$ through which the strict transform of $F_1$ goes, we get $D_{t_2}\sim P_{D_{t_2}}+N_{D_{t_2}}$, which completes the proof.
\end{proof}

Next, we are going to state the main results in this subsection. Recall that $\nu$ is a special exceptional curve valuation whose first component is $\nu_r,$ which is non-positive at infinity and non-minimal with respect to a big divisor $D\sim aF+bM$.

Our results determine the coordinates of the vertices of the Newton-Okounkov bodies $\Delta_\nu(D).$ We divide our study in three cases.

Case A: Either $g^*>1$, or $g^*=1,$ $\nu(\varphi_{F_1})\neq \overline{\beta}_1(\nu)$ and $\nu(\varphi_{M_0})\neq \overline{\beta}_1(\nu)$.

Case B: The value $g^*$ equals $1$ and $\nu(\varphi_{F_1})=\overline{\beta}_1(\nu)$.

Case C: The value $g^*$ equals $1$ and $\nu(\varphi_{M_0})=\overline{\beta}_1(\nu)$.
\medskip

We start with  Case A. Here we can assume that $D$ is also nef (see the paragraph below Theorem \ref{Thm_NOBminimalvaluation}). According with \cite[Theorem 6.4]{LazMus}, by Remark \ref{Remark_muequalmupico}, the Newton-Okounkov body $\Delta_\nu(D)$ coincides with the set 
$$
\{(t,y)\in\mathbb{R}^2\ | \ 0\leq t\leq \hat{\mu}_D(\nu_r) \text{ and } \alpha(t)\leq y \leq \beta (t)\},
$$
where, for all $t\in [0,\hat{\mu}_D(\nu_r)]$, $\alpha(t):=\text{ord}_{p_{r+1}}(N_{D_t}|_{E_r})$ and $\beta(t):=\alpha(t)+ P_{D_t}\cdot E_r;$ here $P_{D_t}$ and $N_{D_t}$ are respectively the positive and negative parts of the divisor $D_{t}=D^*-tE_r$. As a consequence, by Proposition \ref{Prop_Zardecomp_casoespecial}, the points 
\begin{equation*}
\begin{array}{c}
Q_1=\left(\dfrac{b\overline{\beta}_{g+1}(\nu_r)}{\nu_r(\varphi_{F_1})},\dfrac{b\nu_r(\varphi_\eta)}{\nu_r(\varphi_{F_1})}\right) \left(\text{respectively, } Q_1=\left(\dfrac{b\overline{\beta}_{g+1}(\nu_r)}{\nu_r(\varphi_{F_1})},0\right)\right),\\[4mm]
Q_2=Q_1+\left(0, \dfrac{b}{\nu_r(\varphi_{F_1})}\right),\\[4mm]
Q_3=\left(\dfrac{b\overline{\beta}_{g+1}(\nu_r)}{\nu_r(\varphi_{F_1})}+\theta_1^r(D),\dfrac{b\nu_r(\varphi_\eta)+\theta_1^r(D)\nu_\eta(\varphi_{F_1})}{\nu_r(\varphi_{F_1})}\right)\\[4mm]
\left(\text{respectively, } Q_3=\left(\dfrac{b\overline{\beta}_{g+1}(\nu_r)}{\nu_r(\varphi_{F_1})}+\theta_1^r(D),0\right)\right)\text{ and } Q_4=Q_3 + \left(0, \dfrac{b}{\nu_r(\varphi_{F_1})}\right)\\[4mm]
\end{array}
\end{equation*}
are in $\Delta_\nu(D)$ whenever $\theta_1^r(D)\geq0$ and the point $p_{r+1}\in E_\eta\cap E_r$ is satellite (respectively, free). When  $\theta_1^r(D)<0$ and the point $p_{r+1}\in E_\eta\cap E_r$ is satellite (respectively, free), the points are 
\begin{equation*}
\begin{array}{c}
Q_5=\left(\dfrac{a\overline{\beta}_{g+1}(\nu_r)}{\nu_r(\varphi_{M_0})},\dfrac{a\nu_r(\varphi_\eta)}{\nu_r(\varphi_{M_0})}\right)\left(\text{respectively, } Q_5=\left(\dfrac{a\overline{\beta}_{g+1}(\nu_r)}{\nu_r(\varphi_{M_0})},0\right)\right),\\[4mm]
Q_6=Q_5+\left(0, \dfrac{a}{\nu_r(\varphi_{M_0})}\right),\\[4mm]
Q_7=\left(\dfrac{(a+b\delta)\overline{\beta}_{g+1}(\nu_r)-\theta_1^r(D)\nu_r(\varphi_{M_0})}{\nu_r(\varphi_{M_0})+\delta\nu_r(\varphi_{F_1})},\dfrac{(a+b\delta)\nu_r(\varphi_\eta)-\theta_1^r(D)\nu_\eta(\varphi_{M_0})}{\nu_r(\varphi_{M_0})+\delta\nu_r(\varphi_{F_1})}\right)\\[4mm]
\left(\text{respectively, }Q_7=\left(\dfrac{(a+b\delta)\overline{\beta}_{g+1}(\nu_r)-\theta_1^r(D)\nu_r(\varphi_{M_0})}{\nu_r(\varphi_{M_0})+\delta\nu_r(\varphi_{F_1})},0\right)\right)\\[4mm]
\text{ and } Q_8=Q_7 + \left(0, \dfrac{a+b\delta}{\nu_r(\varphi_{M_0})+\delta\nu_r(\varphi_{F_1})}\right).\\[4mm]
\end{array}
\end{equation*}
Notice that $Q_5$ and $Q_6$ may not be well-defined when $\theta_1^r(D)=0$, in fact this happens if $a=\nu_r(\varphi_{M_0})=0$. 

By definition, it also holds that the point $Q_9=(\hat{\mu}_D(\nu_r),\hat{\mu}_D(\nu_\eta))$ (respectively, $Q_9=(\hat{\mu}_D(\nu_r),0)$) when $p_{r+1}$ is satellite (respectively, free) belongs to $\Delta_\nu(D)$. By Theorem \ref{Thm_mupico_nonpositive}, we are able to compute explicitly this point. Now, we state our first main result where we use the symbol $\preccurlyeq$ defined in Section \ref{Subsec_Hirz_mupico}.

\begin{theorem}\label{Thm_NOBodies_specialcase}
Let $\nu$ be a valuation in Case A. With notations as in the previous two paragraphs, the Newton-Okoun\-kov body $\Delta_\nu(D)$ of $D$ with respect to $\nu$ is a quadrilateral if and only if $a\neq 0$ and $\theta_1^r(D)\neq 0$. Otherwise, it is a triangle.

The vertices of the quadrilateral are
\begin{itemize}
\item[(a)] $(0,0),Q_1,Q_3$ (respectively, $Q_5,Q_7$) and $Q_9$ when $\theta_1^r(D)> 0$ (respectively, $\theta_1^r(D)<0$), $p_{r+1}$ is the satellite point $E_\eta\cap E_r$ and $r\not\preccurlyeq \eta.$
\item[(b)] $(0,0),Q_2,Q_4$ (respectively, $Q_6,Q_8$) and $Q_9$ when $\theta_1^r(D)> 0$ (respectively, $\theta_1^r(D)<0$), $p_{r+1}$ is the satellite point $E_\eta\cap E_r$ and $r\preccurlyeq \eta.$
\item[(c)] $(0,0),Q_2,Q_4$ (respectively, $Q_6,Q_8$) and $Q_9$ when $\theta_1^r(D)> 0$ (respectively, $\theta_1^r(D)<0$) and $p_{r+1}$ is a free point.
\end{itemize}

When $\delta >0, a=0$ and $\theta_1^r(D)<0$, $Q_5=Q_6=(0,0)$ and the vertices of the triangle $\Delta_\nu(D)$ are as described in items (a), (b) and (c).

Finally, replacing $\theta_1^r(D)>0$ (or $\theta_1^r(D)<0$) with $\theta_1^r(D)=0$ in items (a), (b) and (c) we obtain the vertices of the triangle $\Delta_\nu(D)$ because $Q_1=Q_3\,(=Q_5=Q_7,\text{ when }\nu(\varphi_{M_0})\neq 0)$ in Case (a) and $Q_2=Q_4\,(=Q_6=Q_8,$ $\text{ when }\nu(\varphi_{M_0})\neq 0)$ otherwise.
\end{theorem}

\begin{proof}
First we will show that $D^2/2$ is the area of the convex sets $\Delta$ and $\Delta'$ defined respectively by the sets of points $\{(0,0),Q_1,Q_2,Q_3,Q_4,Q_9\}$ and $\{(0,0),Q_5,Q_6,$ $Q_7,Q_8,Q_9\}$.

Let us start with $\Delta$. The area of the triangle $(0,0),Q_1$ and $Q_2$ (respectively, $Q_3,Q_4$ and $Q_9$) is 
$$
\dfrac{b^2\overline{\beta}_{g+1}(\nu_r)}{2\nu_r(\varphi_{F_1})^2} 
\text{ \bigg(respectively, } \dfrac{b}{2\nu_r(\varphi_{F_1})}
\Big(\hat{\mu}_D(\nu_r)-\big(\dfrac{b}{\nu_r(\varphi_{F_1})}\overline{\beta}_{g+1}(\nu_r) + \theta_1^r(D)\big)\Big)
\bigg).
$$
The area of the parallelogram $Q_1,Q_2,Q_3$ and $Q_4$ is $\frac{b}{\nu_r(\varphi_{F_1})}\theta_1^r(D).$ Thus, the area of $\Delta$ will be the sum of the above areas, which is 
\begin{align*}
\dfrac{2ab+b^2\delta}{2}= \dfrac{D^2}{2}.
\end{align*}

With respect to $\Delta'$, we have to sum the area of the triangles with vertices $(0,0),Q_5$ and $Q_6$, and  $Q_7,Q_8$ and $Q_9$, with the area of a trapezium  whose vertices are  $Q_5,Q_6,Q_7$ and $Q_8$. The areas of the triangles are equal to $\frac{a^2}{2\nu_r(\varphi_{M_0})^2}\overline{\beta}_{g+1}(\nu_r)$ and 
$$
\dfrac{a+b\delta}{2(\nu_r(\varphi_{M_0})+\delta\nu_r(\varphi_{F_1}))}\left(\hat{\mu}_D(\nu_r)-\left(\dfrac{(a+b\delta)\overline{\beta}_{g+1}(\nu_r)-\theta_1^r(D)\nu_r(\varphi_{M_0})}{\nu_r(\varphi_{M_0})+\delta\nu_r(\varphi_{F_1})}\right)\right).
$$
The length of the parallel sides of the trapezium and the distance between them are
\begin{equation*}
\begin{array}{c}
\dfrac{a}{\nu_r(\varphi_{M_0})}, \ \ \dfrac{a+b\delta}{\nu_r(\varphi_{M_0})+\delta\nu_r(\varphi_{F_1})}\text{ and }\\[4mm] 
\dfrac{-\theta_1^r(D)(\delta\overline{\beta}_{g+1}(\nu_r)+\nu_r(\varphi_{M_0})^2)}{\nu_r(\varphi_{M_0})(\nu_r(\varphi_{M_0})-\delta\nu_r(\varphi_{F_1}))},
\end{array}
\end{equation*}
and the area is 
$$
\dfrac{-\theta_1^r(D)\left((2a+b\delta)\nu_r(\varphi_{M_0})+a\delta\nu_r(\varphi_{F_1})\right)(\delta\overline{\beta}_{g+1}(\nu_r) +\nu_r(\varphi_{M_0})^2)}{2\nu_r(\varphi_{M_0})^2(\nu_r(\varphi_{M_0})+\delta\nu_r(\varphi_{F_1}))^2}.
$$
When summing, the coefficients of $\overline{\beta}_{g+1}(\nu_r)$ are cancelled. Therefore, we only have to sum the following fractions 
\begin{equation*}
\begin{array}{c}
\dfrac{(a+b\delta)\hat{\mu}_D(\nu_r)}{2(\nu_r(\varphi_{M_0})+\delta\nu_r(\varphi_{F_1}))},\dfrac{\theta_1^r(D)(a+b\delta)\nu_r(\varphi_{M_0})}{2(\nu_r(\varphi_{M_0})+\delta\nu_r(\varphi_{F_1}))^2} \text{ and }\\[4mm] 
\dfrac{-\theta_1^r(D)\nu_r(\varphi_{M_0})^2((2a+b\delta b)\nu_r(\varphi_{M_0})+a\delta\nu_r(\varphi_{F_1}))}{2\nu_r(\varphi_{M_0})^2(\nu_r(\varphi_{M_0})+\delta\nu_r(\varphi_{F_1}))^2},
\end{array}
\end{equation*}
giving rise to the desired value $D^2/2$.

Let us show that the defining points of $\Delta$ and $\Delta'$ that do not appear in the items (a), (b) and (c)  belong to the line $L$ which goes through $(0,0)$ and $Q_9$. It is clear that $(0,0),Q_1,Q_3$ (respectively, $Q_5$ and $Q_7$) and $Q_9$ are in $L$ when $\theta_1^r(D)\geq 0$ (respectively, $\theta_1^r(D)<0$) and $p_{r+1}$ is a free point. This corresponds to Item (c).

Now we suppose that $p_{r+1}$ is satellite and $r\preccurlyeq \eta.$ Then, $p_r$ is also a satellite point, $g^*=g$ and, by \cite[Proposition 2.5]{GalMonMoyNic2}, one obtains that 
$$
\nu_r(\varphi_{\eta})=e_{g-1}(\nu_\eta)\overline{\beta}_g(\nu_r)=e_{g-1}(\nu_r)\overline{\beta}_{g}(\nu_r)\dfrac{\overline{\beta}_0(\nu_\eta)}{\overline{\beta}_0(\nu_r)}=\overline{\beta}_{g+1}(\nu_r)\dfrac{\overline{\beta}_0(\nu_\eta)}{\overline{\beta}_0(\nu_r)},
$$
where $e_{g-1}(\nu_i)=\gcd(\overline{\beta}_0(\nu_i),\overline{\beta}_1(\nu_i),\ldots,\overline{\beta}_{g-1}(\nu_i)),$ for $i=r$ or $\eta$. Moreover, by the proof of Lemma 3.9 in \cite{GalMonMoyNic2}, it holds that $e_{g-1}(\nu_\eta)\overline{\beta}_{g}(\nu_r)-e_{g-1}(\nu_r)\overline{\beta}_g(\nu_\eta)=-1$ and then
$$
\nu_r(\varphi_\eta)+1=e_{g-1}(\nu_r)\overline{\beta}_{g}(\nu_\eta)=\overline{\beta}_{g+1}(\nu_r)\dfrac{\overline{\beta}_g(\nu_\eta)}{\overline{\beta}_g(\nu_r)}.
$$
Also, we have that 
$$
\nu_\eta(\varphi_{F_1})=\frac{\overline{\beta}_0(\nu_\eta)}{\overline{\beta}_0(\nu_r)}\nu_r(\varphi_{F_1})
\text{ and }\nu_\eta(\varphi_{M_0})=\frac{\overline{\beta}_0(\nu_\eta)}{\overline{\beta}_0(\nu_r)}\nu_r(\varphi_{M_0}).
$$ 
As a result, it is easy to check that the points $(0,0)$ and $Q_1,Q_3$ (respectively, $Q_5,Q_7$) and $Q_9$ are in the line $L\equiv \overline{\beta}_0(\nu_r)y=\overline{\beta}_0(\nu_\eta)x$ when $\theta_1^r(D)\geq 0$ (respectively, $\theta_1^r(D)< 0$), which corresponds to Item (b). 

A similar reasoning can be applied to the case when $p_{r+1}$ is satellite and $r\not\preccurlyeq \eta$. Notice that in this case $Q_2,Q_4$ (respectively, $Q_6,Q_8$) and $Q_9$ are in the line $L$ when $\theta_1^r(D)\geq 0$ (respectively, $\theta_1^r(D)<0$).

As a consequence of our reasoning $\Delta_\nu(D)$ is a quadrilateral or a triangle.  To conclude the proof we will show that $\Delta_{\nu}(D)$ is a triangle if and only if the conditions in the last two paragraphs of the statement hold. Otherwise, $\Delta_\nu(D)$ will be a quadrilateral.  

Assume, for instance, that $p_{r+1}$ is a satellite point and $r\preccurlyeq \eta.$ Suppose also that $\theta_1^r(D)\geq 0$. In this case, $\Delta_\nu(D)$ is a triangle if and only if one of the following conditions is satisfied: $Q_4$ belongs to the line with equation $\overline{\beta}_g(\nu_\eta)x=\overline{\beta}_g(\nu_r)y$, or $Q_4$ belongs to the line which goes through $Q_2$ and $Q_9$. These two conditions happen if and only if $\theta_1^r(D)=0,$ which proves our statement. Now consider that $\theta_1^r(D)<0$. Here, $\Delta_\nu(D)$ is a triangle if and only if one of the next conditions holds: $Q_8$ belongs to the line with equation $\overline{\beta}_g(\nu_\eta)x=\overline{\beta}_g(\nu_r)y$; $Q_8$ belongs to the line which goes through $Q_6$ and $Q_9$; or $Q_6=(0,0)=Q_5.$ The first and second conditions are true if and only if $\theta_1^r(D)=0$, which is a contradiction because we have supposed that $\theta_1^r(D)$ is negative. The third one happens if and only if $\delta>0$ and $a=0$. This completes the proof after noticing that the remaining cases can be proved analogously.
\end{proof}

An example that corresponds to Statement (a) of the above theorem is the next one.
%We finish this section with an example that corresponds to Statement (a) of the above theorem.

\begin{example}\label{Example_NOB}
Let $p$ be a special point of the Hirzebruch surface $\mathbb{F}_2$ and $\nu_r$  a special divisorial valuation centered at $\mathcal{O}_{\mathbb{F}_2,p}$  whose sequence of maximal contact values is $\{\overline{\beta}_i(\nu_r)\}_{i=0}^3=\{20,28,153,$ $612\}$. Let $\mathcal{C}_{\nu_r}=\{p_i\}_{i=1}^{12}$ (with $p=p_1$) its configuration of infinitely near points and set $F_1$ the fiber which passes through $p$. Suppose that the  strict transform of $M_0$ passes through $p_2$. Then,  $\nu_r(\varphi_{F_1})=20$, $\nu_r(\varphi_{M_0})=28$ and  $2\nu_r(\varphi_{F_1})\nu_r(\varphi_{M_0})+\nu_r(\varphi_{F_1})^2\delta=1920>612=\overline{\beta}_{g+1}(\nu_r)$. So, $\nu_r$ is non-positive at infinity by \cite[Theorem 3.6]{GalMonMor}. 

Let $\nu=\nu_{E_\bullet}$ be the valuation defined by the flag $E_{\bullet}=\{Z=Z_{12}\supset E_{12}\supset\{p_{13}\}\},$ where $p_{13}\in E_{8}\cap E_{12}$. According to Theorem \ref{Thm_NOBodies_specialcase}, $\Delta_\nu(F+2M)$ is a quadrilateral with vertices 
$$
(0,0),Q_5=\left(\dfrac{612}{28},\dfrac{152}{28}\right)\!\!,Q_7=\left(\dfrac{4068}{68},\dfrac{1012}{68}\right) \text{and } Q_9=\left(156,39\right),
$$
since $\nu_r$ is non-minimal with respect to $F+2M$ by Corollary \ref{Cor_nonminimalresptanyD}, $\theta_1^r(F+2M)<0$ and $12\not\preccurlyeq  8$. Figure \ref{Fig_NOB} shows the Newton-Okounkov body $\Delta_\nu(F+2M)$ (in dark) and the triangle $\mathfrak{C}(\nu)\cap \mathfrak{H}_{F+2M}(\nu)$ given in Proposition \ref{Prop_coortriangle}.
\end{example}

\definecolor{aqaqaq}{rgb}{0.6274509803921569,0.6274509803921569,0.6274509803921569}
\definecolor{eqeqeq}{rgb}{0.8784313725490196,0.8784313725490196,0.8784313725490196}
\definecolor{uuuuuu}{rgb}{0.26666666666666666,0.26666666666666666,0.26666666666666666}
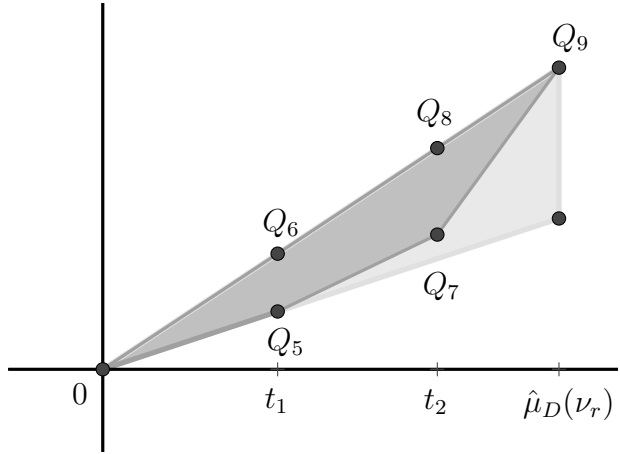
\begin{figure}[ht!]
\begin{center}
\begin{tikzpicture}[line cap=round,line join=round,>=triangle 45,x=1.0cm,y=1.0cm]
\clip(-1.2364583977696633,-1.0882371964892057) rectangle (6.773570558375912,4.857050778397071);
\fill[line width=2.pt,color=eqeqeq,fill=eqeqeq,fill opacity=0.699999988079071] (0.,0.) -- (6.,4.) -- (6.,2.) -- cycle;
\fill[line width=2.pt,color=aqaqaq,fill=aqaqaq,fill opacity=0.550000011920929] (0.,0.) -- (2.3,1.5333333333333332) -- (4.4,2.9333333333333336) -- (6.,4.) -- (4.4,1.7862948958018658) -- (2.3,0.7666666666666666) -- cycle;
\draw [line width=2.pt,color=eqeqeq] (0.,0.)-- (6.,4.);
\draw [line width=2.pt,color=eqeqeq] (6.,4.)-- (6.,2.);
\draw [line width=2.pt,color=eqeqeq] (6.,2.)-- (0.,0.);
\draw [line width=1.2pt,color=aqaqaq] (0.,0.)-- (2.3,1.5333333333333332);
\draw [line width=1.2pt,color=aqaqaq] (2.3,1.5333333333333332)-- (4.4,2.9333333333333336);
\draw [line width=1.2pt,color=aqaqaq] (4.4,2.9333333333333336)-- (6.,4.);
\draw [line width=1.2pt,color=aqaqaq] (6.,4.)-- (4.4,1.7862948958018658);
\draw [line width=1.2pt,color=aqaqaq] (4.4,1.7862948958018658)-- (2.3,0.7666666666666666);
\draw [line width=2.pt,color=aqaqaq] (2.3,0.7666666666666666)-- (0.,0.);
\draw [line width=1.2pt] (0.,-1.0882371964892057) -- (0.,4.857050778397071);
\draw [line width=1.2pt,domain=-1.2364583977696633:6.773570558375912] plot(\x,{(-0.-0.*\x)/1.});
\draw (-0.548210937186771,-0.04854397920613749) node[anchor=north west] {$0$};
\draw (1.9851254602779174,-0.10711824496856387) node[anchor=north west] {$t_1$};
\draw (4.079154968008845,-0.12176181140917046) node[anchor=north west] {$t_2$};
\draw (5.382432074219002,-0.10711824496856387) node[anchor=north west] {$\hat{\mu}_D(\nu_r)$};
\draw (1.9558383342956667,2.265139518409704) node[anchor=north west] {$Q_6$};
\draw (3.976650027070967,3.7294961624703635) node[anchor=north west] {$Q_8$};
\draw (5.763164711988262,4.725258680431612) node[anchor=north west] {$Q_9$};
\draw (4.064511405017719,1.4304562312951283) node[anchor=north west] {$Q_7$};
\draw (2.014412586260168,0.6836343428241921) node[anchor=north west] {$Q_5$};
\begin{scriptsize}
\draw [fill=uuuuuu] (0.,0.) circle (2.5pt);
\draw [fill=uuuuuu] (6.,4.) circle (2.5pt);
\draw [fill=uuuuuu] (6.,2.) circle (2.5pt);
\draw [fill=uuuuuu] (2.3,0.7666666666666666) circle (2.5pt);
\draw [fill=uuuuuu] (2.3,1.5333333333333332) circle (2.5pt);
\draw [fill=uuuuuu] (4.4,2.9333333333333336) circle (2.5pt);
\draw [fill=uuuuuu] (4.4,1.7862948958018658) circle (2.5pt);
\draw [color=uuuuuu] (2.3,0.)-- ++(-2.5pt,0 pt) -- ++(5.0pt,0 pt) ++(-2.5pt,-2.5pt) -- ++(0 pt,5.0pt);
\draw [color=uuuuuu] (4.4,0.)-- ++(-2.5pt,0 pt) -- ++(5.0pt,0 pt) ++(-2.5pt,-2.5pt) -- ++(0 pt,5.0pt);
\draw [color=uuuuuu] (6.,0.)-- ++(-2.5pt,0 pt) -- ++(5.0pt,0 pt) ++(-2.5pt,-2.5pt) -- ++(0 pt,5.0pt);
\end{scriptsize}
\end{tikzpicture}
\caption{$\Delta_\nu(F+2M)$ and  $\mathfrak{C}(\nu)\cap \mathfrak{H}_{F+2M}(\nu)$ in Example \ref{Example_NOB}.}\label{Fig_NOB}
\end{center}
 
\end{figure}

Now we are going to determine the Newton-Okounkov $\Delta_\nu(D)$ for valuations $\nu$ in Case B introduced before Theorem \ref{Thm_NOBodies_specialcase}. That is, we assume that $g^*=1$ and $\nu(\varphi_{F_1})=\overline{\beta}_1(\nu)$. Note that, in this case, it could happen that $\nu(\varphi_{M_0})=(0,0)$ and then $\theta_1^r(D)=a\nu_r(\varphi_{F_1})\geq 0$. In addition, we can assume that $D$ is a big and nef divisor because when $D$ is big and not nef, $\Delta_\nu(D)$ can be computed as we explained in  paragraph under Theorem \ref{Thm_NOBminimalvaluation}. Following \cite[Theorem 6.4]{LazMus} and Proposition \ref{Prop_Zardecomp_casoespecial}, if $p_{r+1}\in E_r\cap E_\eta$ is a satellite point and $\theta_1^r(D)\geq 0$ (respectively, $\theta_1^r(D)<0$), then the points $Q_1,Q_2,Q_3,Q_4$ (respectively, $Q_5,Q_6,Q_7,Q_8$) and $Q_9$ described before Theorem \ref{Thm_NOBodies_specialcase} for the satellite case belong to $\Delta_\nu(D)$. Otherwise ($p_{r+1}$ is free), the points 
\begin{equation*}
\begin{array}{c}
Q_1=\left(\dfrac{b\overline{\beta}_{g+1}(\nu_r)}{\nu_r(\varphi_{F_1})},0\right),
Q_2=Q_1+\left(0, \dfrac{b}{\nu_r(\varphi_{F_1})}\right),\\[4mm]
Q_3=\left(\dfrac{b\overline{\beta}_{g+1}(\nu_r)}{\nu_r(\varphi_{F_1})}+\theta_1^r(D),\dfrac{\theta_1^r(D)}{\nu_r(\varphi_{F_1})}\right)\text{ and } Q_4=Q_3 + \left(0, \dfrac{b}{\nu_r(\varphi_{F_1})}\right)
\end{array}
\end{equation*}
(respectively, $Q_5,Q_6,Q_7,Q_8$ given before Theorem \ref{Thm_NOBodies_specialcase} for the free case) and $$Q_9=(\hat{\mu}_D(\nu_r), a+b\delta)$$ are in $\Delta_\nu(D)$ if $\theta_1^r(D)\geq 0$ (respectively, $\theta_1^r(D)<0$).
\begin{theorem}\label{Thm_NOBcasoespecial_g*=1F}
Let $\nu$ be a valuation in Case B. With notations as in the previous paragraph, the Newton-Okounkov body $\Delta_\nu(D)$ of $D$ with respect to $\nu$ is a quadrilateral if and only if $a\neq 0$. Otherwise, it is a triangle.  
\begin{itemize}
\item[(a)] When $\nu(\varphi_{M_0})=(0,0)$, the vertices of the quadrilateral are
\begin{itemize}
\item[(a.1)] $(0,0),Q_2,Q_4$ (respectively, $Q_1,Q_3$) and $Q_9$ if  $p_{r+1}$ is the satellite point $E_\eta\cap E_r$ and $r\not\preccurlyeq \eta,$ (respectively, $r\preccurlyeq \eta$). 
\item[(a.2)] $(0,0),Q_1,Q_3$ and $Q_9$ whenever $p_{r+1}$ is a free point.
\end{itemize}

In addition, if $\delta >0$ and $a=0$, then the vertices of the triangle $\Delta_\nu(D)$ are the above ones, where $Q_1=Q_3$ and $Q_2=Q_4$.
\item[(b)] When $\nu(\varphi_{M_0})\neq (0,0)$, the vertices of the quadrilateral are
\begin{itemize}
\item[(b.1)] $(0,0),Q_2,Q_3$ (respectively, $Q_5,Q_8$) and $Q_9$ if $\theta_1^r(D)\geq 0$ (respectively, $\theta_1^r(D)< 0$), $p_{r+1}$ is the satellite point $E_\eta\cap E_r$ and $r\not\preccurlyeq \eta.$
\item[(b.2)] $(0,0),Q_1,Q_4$ (respectively, $Q_6,Q_7$) and $Q_9$ if $\theta_1^r(D)\geq 0$ (respectively, $\theta_1^r(D)< 0$), $p_{r+1}$ is the satellite point $E_\eta\cap E_r$ and $r\preccurlyeq \eta.$
\item[(b.3)] $(0,0),Q_1,Q_4$ (respectively, $Q_6,Q_7$) and $Q_9$ if $\theta_1^r(D)\geq 0$ (respectively, $\theta_1^r(D)< 0$) and $p_{r+1}$ is a free point.
\end{itemize}

Moreover, if $\delta >0$ and $a=0,$ the vertices of the triangle $\Delta_\nu(D)$ are the above ones where $Q_5=(0,0)=Q_6$.
\end{itemize} 
\end{theorem}
\begin{proof}
Consider the convex sets defined by the points  $\{(0,0),Q_1,Q_2,Q_3,Q_4,Q_9\}$ and $\{(0,0),Q_5,Q_6,Q_7,Q_8,Q_9\}$. Reasoning as in the proof of Theorem  \ref{Thm_NOBodies_specialcase}, we deduce that the area of both sets is $D^2/2$.

To prove items (a.1) and (a.2), it suffices to check that the points defining $\Delta_\nu(D)$ that do not appear in the statement belong to the line $L$ that passes through $(0,0)$ and $Q_9.$ $L$ is defined by the equation $L\equiv \overline{\beta}_{g^*}(\nu_r)y=\overline{\beta}_{g^*}(\nu_\eta)x$ (respectively, $L\equiv y=x/\overline{\beta}_{g+1}(\nu_r)$) when $p_{r+1}$ is a satellite point (respectively, $p_{r+1}$ is a free point) because  
\begin{equation}\label{Cond_Fg*}
\nu_\eta(\varphi_{F_1})=\nu_r(\varphi_{F_1})\dfrac{\overline{\beta}_{g^*}(\nu_\eta)}{\overline{\beta}_{g^*}(\nu_r)}\ \ \left(\text{respectively, } \nu_r(\varphi_{F_1})=\overline{\beta}_{g+1}(\nu_r)\right)
\end{equation}
if $p_{r+1}$ is a satellite point (respectively, $p_{r+1}$ is a free point). To verify the point alignment we can use \cite[Proposition 2.5 and Lemma 3.9]{GalMonMoyNic2}, which prove  that
\begin{equation}\label{Cond_betabarrag*}
\nu_r(\varphi_\eta)=\overline{\beta}_{g+1}(\nu_r)\dfrac{\overline{\beta}_{g^*}(\nu_\eta)}{\overline{\beta}_{g^*}(\nu_r)}\ \ \left(\text{respectively, } \nu_r(\varphi_\eta)+1=\overline{\beta}_{g+1}(\nu_r)\dfrac{\overline{\beta}_{g^*}(\nu_\eta)}{\overline{\beta}_{g^*}(\nu_r)}\right)
\end{equation}
if $r\not\preccurlyeq \eta$ (respectively, $r\preccurlyeq \eta$). 

To conclude the proof, we only show Item (b.1) since the remaining items (b.2) and (b.3) run similarly. First, we suppose that $\theta_1^r(D)\geq 0$. By \eqref{Cond_Fg*} and \eqref{Cond_betabarrag*}, the points  $(0,0),Q_1$ and $Q_3$ belong to the line with equation $\overline{\beta}_{g^*}(\nu_r)y=\overline{\beta}_{g^*}(\nu_\eta)x$. Moreover, it is easily seen that $Q_4$ belongs to the line which goes through $Q_2$ and $Q_9$. Finally, the point $Q_9$ does not belong neither to the line with equation $\overline{\beta}_{0}(\nu_r)y=\overline{\beta}_{0}(\nu_\eta)x$ nor to that with equation $\overline{\beta}_{g^*}(\nu_r)y=\overline{\beta}_{g^*}(\nu_\eta)x$, which finishes the proof in this case $\theta_1^r(D)\geq 0$.

Only remains to assume that $\theta_1^r(D)<0$, then 
\begin{equation*}
\nu_\eta(\varphi_{M_0})=\nu_r(\varphi_{M_0})\dfrac{\overline{\beta}_{0}(\nu_\eta)}{\overline{\beta}_{0}(\nu_r)}\text{ and }\nu_r(\varphi_\eta)+1=\overline{\beta}_{g+1}(\nu_r)\dfrac{\overline{\beta}_{0}(\nu_\eta)}{\overline{\beta}_{0}(\nu_r)}.
\end{equation*}   
Therefore, $(0,0),Q_6$ and $Q_8$ belong to the line with equation $\overline{\beta}_{0}(\nu_r)y=\overline{\beta}_{0}(\nu_\eta)x$. Moreover $Q_7$ is in the line which goes through $Q_5$ and $Q_9,$ which completes the proof. 
\end{proof}

We finish this section by describing the Newton-Okounkov body $\Delta_\nu(D)$ in Case C introduced before Theorem \ref{Thm_NOBodies_specialcase}. Then,  suppose that $g^*=1$ and $\nu(\varphi_{M_0})=\overline{\beta}_1(\nu)$. Here, we can assume that $D$ is a big and nef divisor except for the case when all the points $\{p_i\}_{i=1}^{r+1}$ are free. In this last situation, $p_{r+1}\in\text{supp}( N_{D^*})$ if and only if $D$ is big and not nef.

First, assume that $D$ is big and nef. According \cite[Theorem 6.4]{LazMus} and Proposition \ref{Prop_Zardecomp_casoespecial}, when $p_{r+1}\in E_r\cap E_\eta$ is a satellite point and $\theta_1^r(D)\geq 0$ (respectively, $\theta_1^r(D)<0$) the points $Q_1,Q_2,Q_3,Q_4$ (respectively, $Q_5,Q_6,Q_7,Q_8$) and $Q_9$ described before Theorem \ref{Thm_NOBodies_specialcase} for the satellite case are in $\Delta_\nu(D)$.

When $p_{r+1}$ is free, the points 
\begin{equation*}
\begin{array}{c}
Q_5=\left(\dfrac{a\overline{\beta}_{g+1}(\nu_r)}{\nu_r(\varphi_{M_0})},0\right),
Q_6=Q_5+\left(0, \dfrac{a}{\nu_r(\varphi_{M_0})}\right),\\[4mm]
Q_7=\left(\dfrac{(a+b\delta)\overline{\beta}_{g+1}(\nu_r)-\theta_1^r(D)\nu_r(\varphi_{M_0})}{\nu_r(\varphi_{M_0})+\delta\nu_r(\varphi_{F_1})},\dfrac{-\theta_1^r(D)}{\nu_r(\varphi_{M_0})+\delta\nu_r(\varphi_{F_1})}\right),\\[4mm]
Q_8=Q_7 + \left(0, \dfrac{a+b\delta}{\nu_r(\varphi_{M_0})+\delta\nu_r(\varphi_{F_1})}\right)\\[4mm]
\end{array}
\end{equation*}
(respectively, $Q_1,Q_2,Q_3,Q_4$ provided before Theorem \ref{Thm_NOBodies_specialcase} for the free case) and $Q_9=(\hat{\mu}_D(\nu_r),b)$ belong to $\Delta_\nu(D)$ if $\theta_1^r(D)< 0$ (respectively, $\theta_1^r(D)\geq 0$).

Finally, assume that $D$ is big and not nef and all the points in $\{p_i\}_{i=1}^{r+1}$ are free. Recall that these assumptions are equivalent to the fact that $p_{r+1}\in\text{supp}( N_{D^*})$ (see the paragraph after Theorem \ref{Thm_NOBminimalvaluation}). Then,  the points 
\begin{equation*}
\begin{array}{c}
P_1=\left(\dfrac{-a\nu_r(\varphi_{M_0})}{\delta},\dfrac{-a}{\delta}\right),\\[4mm]
 P_2=\left(\dfrac{(a+b\delta)\overline{\beta}_{g+1}(\nu_r)-\theta_1^r(D)\nu_r(\varphi_{M_0})}{\nu_r(\varphi_{M_0})+\delta\nu_r(\varphi_{F_1})},\dfrac{-\theta_1^r(D)}{\nu_r(\varphi_{M_0})+\delta\nu_r(\varphi_{F_1})}\right),\\[4mm]
P_3=P_2 + \left(0, \dfrac{a+b\delta}{\nu_r(\varphi_{M_0})+\delta\nu_r(\varphi_{F_1})}\right) \text{ and } P_4=(\hat{\mu}_D(\nu_r),b)\\[4mm]
\end{array}
\end{equation*}
are in $\Delta_\nu(D)$.

Next, we state our result for Case C, where, as mentioned, $D$ is big and nef except when $p_{r+1}\in\text{supp}(N_{D^*})$. We recall that the Newton-Okounkov bodies $\Delta_\nu(D)$ for the remaining cases where $D$ is big but not nef can be reduced to the big and nef situation (see the paragraph below Theorem \ref{Thm_NOBminimalvaluation}).

\begin{theorem}\label{Thm_NOBcasoespecial_g*=1M0}
Let $\nu$ be a valuation in Case C. With assumptions and notations as in the previous paragraphs, the Newton-Okounkov body $\Delta_\nu(D)$ of $D$ with respect to $\nu$ is a quadrilateral if and only if $a\neq 0$ and $D$ is nef. Otherwise, it is a triangle.  
\begin{itemize}
\item[(a)] When $D$ is a big and nef divisor, the vertices of the quadrilateral are
\begin{itemize}
\item[(a.1)] $(0,0),Q_1,Q_4$ (respectively, $Q_6,Q_7$) and $Q_9$ if $\theta_1^r(D)\geq 0$ (respectively, $\theta_1^r(D)< 0$), $p_{r+1}$ is the satellite point $E_\eta\cap E_r$ and $r\not\preccurlyeq \eta.$
\item[(a.2)] $(0,0),Q_2,Q_3$ (respectively, $Q_5,Q_8$) and $Q_9$ if $\theta_1^r(D)\geq 0$ (respectively, $\theta_1^r(D)< 0$), $p_{r+1}$ is the satellite point $E_\eta\cap E_r$ and $r\preccurlyeq \eta.$
\item[(a.3)] $(0,0),Q_2,Q_3$ (respectively, $Q_5,Q_8$) and $Q_9$ if $\theta_1^r(D)\geq 0$ (respectively, $\theta_1^r(D)< 0$) and $p_{r+1}$ is a free point.
\end{itemize}
Moreover, if $\delta >0$ and $a=0,$ the vertices of the triangle $\Delta_\nu(D)$ are the above ones where $Q_5=(0,0)=Q_6$. 
\item[(b)] If $D$ is big but not nef and all the points in $\{p_{i}\}_{i=1}^{r+1}$ are free, the vertices of the triangle $\Delta_\nu(D)$ are $P_1,P_3$ and $P_4$.
\end{itemize} 
\end{theorem}
\begin{proof}
Item (a) follows as in Theorem \ref{Thm_NOBcasoespecial_g*=1F} (b) after considering that $$\nu_\eta(\varphi_{M_0})\overline{\beta}_{g^*}(\nu_r)=\nu_r(\varphi_{M_0})\overline{\beta}_{g^*}(\nu_\eta)\text{ and } \nu_\eta(\varphi_{F_1})\overline{\beta}_{0}(\nu_r)=\nu_r(\varphi_{F_1})\overline{\beta}_{0}(\nu_\eta).$$

Now we are going to prove Item (b). For a start, the area of the convex set $\Delta$ generated by $P_1,P_2,P_3$ and $P_4$ is $P_{D^*}^2/2$. Indeed, the area of the triangle generated by $P_1,P_2$ and $P_3$ (respectively, $P_2,P_3$ and $P_4$) is 
$$
\dfrac{(a+b\delta)\left(\frac{(a+b\delta)\overline{\beta}_{g+1}(\nu_r)-\theta_1^r(D)\nu_r(\varphi_{M_0})}{\nu_r(\varphi_{M_0})+\delta\nu_r(\varphi_{F_1})}-\frac{-a\nu_r(\varphi_{M_0})}{\delta}\right)}{2(\nu_r(\varphi_{M_0})+\delta\nu_r(\varphi_{F_1}))}
$$
$$
\left(\text{respectively, } \dfrac{(a+b\delta)\left(\hat{\mu}_{D}(\nu_r)-\frac{(a+b\delta)\overline{\beta}_{g+1}(\nu_r)-\theta_1^r(D)\nu_r(\varphi_{M_0})}{\nu_r(\varphi_{M_0})+\delta\nu_r(\varphi_{F_1})}\right)}{2(\nu_r(\varphi_{M_0})+\delta\nu_r(\varphi_{F_1}))}\right).
$$
Therefore, the area of $\Delta$ is the sum of the two above areas, which is 
\begin{align*}
\dfrac{(a+b\delta)\left(\hat{\mu}_{D}(\nu_r)-\frac{-a\nu_r(\varphi_{M_0})}{\delta}\right)}{2(\nu_r(\varphi_{M_0})+\delta\nu_r(\varphi_{F_1}))}&=\dfrac{\left(a+b\delta\right)\left(b+\frac{a}{\delta}\right)\left(\nu_r(\varphi_{M_0})+\delta\nu_r(\varphi_{F_1})\right)}{2(\nu_r(\varphi_{M_0})+\delta\nu_r(\varphi_{F_1}))}\\ &=\dfrac{\left(\left(b+\frac{a}{\delta}\right)M^*\right)^2}{2}=\dfrac{P_{D^*}^2}{2}.
\end{align*}
Finally, $P_2$ belongs to the line which goes through $P_1$ and $P_4,$ and $P_4$ is not in the line with equation $\overline{\beta}_{g+1}(\nu_r)y=x,$ which completes the proof of Item (b).
\end{proof}

\begin{remark}
Notice that the particular cases $\delta=1$, $a=0$ and $\theta_1^r(D)<0$ in Theorems \ref{Thm_NOBodies_specialcase}, \ref{Thm_NOBcasoespecial_g*=1F} and \ref{Thm_NOBcasoespecial_g*=1M0} provide the Newton-Okounkov body described in \cite[Corollary 5.2]{GalMonMoyNic2}. This holds because $\mathbb{F}_1$ is the blow-up of the projective plane $\mathbb{P}^2$ at a point, and the special section, in this case, is the exceptional divisor.
\end{remark}

\subsection{Newton-Okounkov bodies with respect to non-positive at infinity non-special valuations}

In this last subsection, we complete Subsection \ref{Subsec_NOofNPIvalspec} by considering non-special valuations. Denote by $\nu$ a non-positive at infinity non-special exceptional curve  valuation whose first component is $\nu_r$ and assume that $D\sim aF+bM$ is a big and nef divisor on $\mathbb{F}_\delta$ (since $p_1$ is a general point by Definition \ref{Def_spe_nspe_val}). We will use the notation $\theta_2^r(D)$ for the value  $a\nu_r(\varphi_{F_1})-b\big(\nu_r(\varphi_{M_1})-\delta\nu_r(\varphi_{F_1})\big)$, where $F_1$ and $M_1$ are as defined below Definition \ref{Def_spe_nspe_val}. The following results translate to the non-special case what happens in Subsection \ref{Subsec_NOofNPIvalspec} for the special one.

\begin{lemma}\label{Lemma_positivepartnef_nonspecialcase}
Let $\nu_r$ be a non-positive at infinity non-special divisorial valuation of $\mathbb{F}_\delta$. Set $D$ and $\theta_2^r(D)$ as above. Then, the divisor 
$$
D_3=D^*-\dfrac{b}{\nu_r(\varphi_{F_1})}\sum_{i=1}^r\nu_r(\mathfrak{m}_i)E_i^* \left(\!\!\text{respectively, } D_4=D^*-\dfrac{a+b\delta}{\nu_r(\varphi_{M_1})}\sum_{i=1}^r\nu_r(\mathfrak{m}_i)E_i^*\right)
$$
is nef when $\theta_2^r (D)\geq 0$ (respectively, $\theta_2^r (D)<0$).
\end{lemma}
\begin{proof}
We are going to show that $D_4$ is nef when $\theta_2^r (D)< 0$. The fact that the divisor $D_3$ is nef follows  from a similar reasoning as that used in Lemma \ref{Lemma_partepositivanef_casoespecial}.

Write 
$$
\Delta_r:=(\nu_r(\varphi_{M_1})-\delta\nu_r(\varphi_{F_1}))F^*+\nu_r(\varphi_{F_1})M^* - \sum_{i=1}^r\nu_r(\mathfrak{m}_i)E_i^*
\text{ and}
$$ 
$$
\Gamma_r:=\nu_r(\varphi_{M_1})M^* - \delta\sum_{i=1}^r\nu_r(\mathfrak{m}_i)E_i^*.
$$ 
Both divisors are nef by \cite[Theorem 4.8]{GalMonMor} and this concludes the proof since 
$$
D_4\sim\dfrac{a}{\nu_r(\varphi_{M_1})-\delta\nu_r(\varphi_{F_1})}\Delta_r + \dfrac{-\theta_2^r(D)}{\nu_r(\varphi_{M_1})-\delta\nu_r(\varphi_{F_1})}\Gamma_r
$$
and $-\theta_2^r(D)>0$.
\end{proof}

The following result can be proved reasoning as in the proof of Lemma \ref{Lemma_t_ibelongstoT_casoespecial}. Notice that we are considering a non-special divisorial valuation whose non-positivity at infinity can be checked with the inequality below Definition \ref{Def_NPI_spe_nspe_val}. Recall that we are considering a big and nef divisor $D\sim aF+bM$ on $\mathbb{F}_\delta$. We will also use the value $\theta_2^r(D)$. 

\begin{lemma}\label{Lemma_t_ibelongstoT_casonoespecial}
Let $\nu_r$ be a non-positive at infinity non-special divisorial va\-luation of $\mathbb{F}_\delta$. Then, the rational numbers 
\[
t_5=\dfrac{b}{\nu_r(\varphi_{F_1})}\overline{\beta}_{g+1}(\nu_r) \text{ and } t_6=\dfrac{b}{\nu_r(\varphi_{F_1})}\overline{\beta}_{g+1}(\nu_r)+\theta_2^r (D)
\]
\[
\left(\text{respectively, }
t_7=\dfrac{a+b\delta}{\nu_r(\varphi_{M_1})}\overline{\beta}_{g+1}(\nu_r)
\text{ and } 
t_8=\dfrac{a\overline{\beta}_{g+1}(\nu_r) - \nu_r(\varphi_{M_1})\theta_2^r (D)}{\nu_r(\varphi_{M_1})-\delta \nu_r(\varphi_{F_1})}\right)
\]
belong to the set $T_{D,\nu_r}:=\{t\in\mathbb{Q} \ | \ 0\leq t \leq \hat{\mu}_D(\nu_r)\}$ when $\theta^r_2(D)\geq 0$ (respectively, $\theta^r_2(D)<0$).  
\end{lemma}

\begin{remark}
As in the special divisorial valuation case, if $\nu_r$ is minimal with respect to $D$, by Theorem \ref{Thm_mupico_nonpositive} and Corollary \ref{Cor_minimalcase}, one gets
$$
\hat{\mu}_D(\nu_r)=b\overline{\beta}_{g+1}(\nu_r)/\nu_r(\varphi_{F_1})=t_5=t_6=t_7=t_8.
$$ 
Otherwise, Lemma  \ref{Lemma_t_ibelongstoT_casonoespecial} provides two values, $t_5$ and $t_6$ (respectively, $t_7$ and $t_8$) when $\theta_2^r(D)\geq 0$ (respectively, $\theta_2^r(D)<0$). When $\theta_2^r(D)=0$, one has that $\hat{\mu}_D(\nu_r)>t_5=t_6=t_7=t_8$, and if $a=0$ and $\theta_2^r(D)<0,$ then $t_8=\hat{\mu}_D(\nu_r)$. Moreover, if the equality $2\nu_r(\varphi_{M_1})\nu_r(\varphi_{F_1})-\delta\nu_r(\varphi_{F_1})^2 = \overline{\beta}_{g+1}(\nu_r)$ holds, we obtain that $t_6=\hat{\mu}_D(\nu_r)$ (respectively, $t_8=\hat{\mu}_D(\nu_r)$) whenever $\theta_2^r(D)>0$ (respectively, $\theta_2^r(D)<0$).
\end{remark}

Reasoning as in Lemma \ref{Lemma_Ntdeterminesmatrixnegativedefinite_casoespecial}   one proves that the divisors $D_3$ and $D_4$ in Lemma \ref{Lemma_positivepartnef_nonspecialcase} are big. Moreover,  $D_3\cdot \tilde{F}_1=0,D_4\cdot \tilde{M}_1=0$, and $D_3\cdot E_i=0$ and $D_4\cdot E_i=0$, for $1\leq i\leq r-1$. As a consequence, one gets the following result.

\begin{lemma}\label{Lemma_Ntdeterminesmatrisdefinednegative_caso noespecial}
Let $\nu_r$ be a divisorial valuation and $D$ a divisor as in Lemma \ref{Lemma_positivepartnef_nonspecialcase}. Assume also that $\nu_r$ is non-minimal with respect to $D$. The intersection matrix determined by the set of divisors $\{\tilde{F}_1,E_1,\ldots,E_{r-1}\}$ (respectively, $\{\tilde{M}_1,E_1,\ldots,E_{r-1}\}$) is negative definite.
\end{lemma}

Our upcoming proposition considers a valuation $\nu_r$ and a divisor $D$ as stated before Lemma \ref{Lemma_positivepartnef_nonspecialcase} and determines the Zariski decomposition of the divisors $D^*-t_iE_r,$ $5\leq i\leq 8,$ where $t_i$ are the rational numbers defined in Lemma \ref{Lemma_t_ibelongstoT_casonoespecial}. We will use the above defined value $\theta_2^r(D)$ and the divisors $D_3,D_4$ and $\Delta_r=(\nu_r(\varphi_{M_1})-\delta\nu_r(\varphi_{F_1}))F^*+\nu_r(\varphi_{F_1})M^*-\sum_{i=1}^r\nu_r(\mathfrak{m}_i)E_i^*$ given in Lemma \ref{Lemma_positivepartnef_nonspecialcase} and its proof.

\begin{proposition}\label{Prop_Zardecomp_caso_noespecial}
The following statements hold.
\begin{itemize}
\item[(a)] The positive and negative parts of the Zariski decomposition of the divisor $D_{t_5}=D^* -t_5E_r$ (respectively, $D_{t_6}=D^*-t_6E_r$) are 
\begin{equation*}
\begin{array}{c}
P_{D_{t_5}}\sim D_3 \ \text{ and } \  N_{D_{t_5}} = \dfrac{b}{\nu_r(\varphi_{F_1})}\displaystyle\sum_{i=1}^{r-1}\nu_r(\varphi_i)E_i\\[4mm]
\Bigg(\text{respectively, } P_{D_{t_6}}\sim \dfrac{b}{\nu_r(\varphi_{F_1})}\Delta_r  \text{ and }\\[4mm]
N_{D_{t_6}}=\dfrac{\theta_2^r(D)}{\nu_r(\varphi_{F_1})}\tilde{F}+\displaystyle\sum_{i=1}^{r-1} \dfrac{b\nu_r(\varphi_i) + \theta_2^r(D)\nu_i(\varphi_{F_1}) }{\nu_r(\varphi_{F_1})}E_i\Bigg),
\end{array}
\end{equation*}
when  $\theta_2^r(D)\geq 0$.
\item[(b)] The positive and negative parts of the Zariski decomposition of $D_{t_7}=D^* -t_7E_r$ (respectively, $D_{t_8}=D^*-t_8E_r$) are 
\begin{equation*}
\begin{array}{c}
P_{D_{t_7}}\sim D_4 \ \text{ and } \  N_{D_{t_7}} = \dfrac{a+b\delta}{\nu_r(\varphi_{M_1})}\displaystyle\sum_{i=1}^{r-1}\nu_r(\varphi_i)E_i\\[4mm] 
\Bigg(\text{respectively, } P_{D_{t_8}}\sim \dfrac{a}{\nu_r(\varphi_{M_1})-\delta \nu_r(\varphi_{F_1})}\Delta_r  \text{ and }\\[4mm] 
N_{D_{t_8}}=\dfrac{-\theta_2^r(D)}{\nu_r(\varphi_{M_1})-\delta\nu_r(\varphi_{F_1})}\tilde{M}_1
+\displaystyle\sum_{i=1}^{r-1}\dfrac{a\nu_r(\varphi_i)-\theta_2^r(D)\nu_i(\varphi_{M_1})}{\nu_r(\varphi_{M_1})-\delta\nu_r(\varphi_{F_1})}E_i\Bigg),
\end{array}
\end{equation*}
when $\theta_2^r(D)<0$.
\end{itemize}
\end{proposition}

\begin{proof}
We are going to prove Statement (b). A proof for $(a)$  runs similarly.  On the one hand, the components of the divisor $N_{D_{t_7}}$ determine a negative definite intersection matrix. On the other hand, the divisor $P_{D_{t_7}}$ is nef by Lemma \ref{Lemma_positivepartnef_nonspecialcase} and orthogonal to each component of $N_{D_{t_7}}$ by the proximity equalities. So, $P_{D_{t_7}} + N_{D_{t_7}}$ gives the Zariski decomposition of $D_{t_7}$. 

Let us show the claim for $D_{t_8}$. By Lemma \ref{Lemma_Ntdeterminesmatrisdefinednegative_caso noespecial}, the components of $N_{D_{t_8}}$ de\-ter\-mine a negative definite intersection matrix and, by \cite[Proposition 4.1 and Theorem 4.8]{GalMonMor}, the divisor $P_{D_{t_8}}$ is nef and orthogonal to each component of $N_{D_{t_8}}$. Finally, we are going to see that $P_{D_{t_8}}+N_{D_{t_8}}\sim D_{t_8}$, which completes the proof. Indeed, let $p_{i_{M_1}}$ be the last point in the configuration of infinitely near points $\mathcal{C}_{\nu_r}$ of the valuation $\nu_r$ through which the strict transform of $M_1$ goes. Since $\tilde{M}_1\sim M^*-\sum_{i=1}^{i_{M_1}}E_i^*$, it holds that 
$$
\dfrac{a(\Delta_r+\sum_{i=1}^{r-1}\nu_r(\varphi_i)E_i) + \theta_2^r(D)M^*}{\nu_r(\varphi_{M_1})-\delta\nu_r(\varphi_{F_1})}\sim D-\dfrac{a\overline{\beta}_{g+1}(\nu_r)}{\nu_r(\varphi_{M_1})-\delta\nu_r(\varphi_{F_1})}E_r. 
$$  
In addition, 
$$
\dfrac{-\theta_2^r(D)}{\nu_r(\varphi_{M_1})-\delta\nu_r(\varphi_{M_1})}\left(\sum_{i=1}^{r-1}\nu_i(\varphi_{M_1})E_i-\sum_{i=1}^{i_{M_1}}E_i^*\right)=\dfrac{-\theta_2^r(D)\nu_r(\varphi_{M_1})}{\nu_r(\varphi_{M_1})-\delta\nu_r(\varphi_{F_1})}E_r ,
$$
and the result follows after summing both expressions. 
\end{proof}

We conclude our paper by determining the vertices of the Newton-Okoun\-kov body $\Delta_\nu(D)$, where $D$ and $\nu$ are as in the paragraph before Lemma \ref{Lemma_positivepartnef_nonspecialcase}. Recall that $\nu_r$ is the first component of $\nu.$ 
We again divide our description of $\Delta_\nu(D)$ in two cases:

Case D: Either $g^*>1$ or $g^*=1$ and $\nu(\varphi_{M_1})\neq \overline{\beta}_1(\nu)$.

Case E: The value $g^*$ equals $1$ and $\nu(\varphi_{M_1})= \overline{\beta}_1(\nu)$.

\medskip

Let us start with the case D. Arguing as before Theorem \ref{Thm_NOBodies_specialcase}, the points 
\begin{equation*}
\begin{array}{c}
Q_{10}=\left(\dfrac{b\overline{\beta}_{g+1}(\nu_r)}{\nu_r(\varphi_{F_1})},\dfrac{b\nu_r(\varphi_\eta)}{\nu_r(\varphi_{F_1})}\right) \left(\text{respectively, } Q_{10}=\left(\dfrac{b\overline{\beta}_{g+1}(\nu_r)}{\nu_r(\varphi_{F_1})},0\right)\right),\\[4mm]
Q_{11}=Q_{10}+\left(0, \dfrac{b}{\nu_r(\varphi_{F_1})}\right),\\[4mm]
Q_{12}=\left(\dfrac{b\overline{\beta}_{g+1}(\nu_r)}{\nu_r(\varphi_{F_1})}+\theta_2^r(D),\dfrac{b\nu_r(\varphi_\eta)+\theta_2^r(D)\nu_\eta(\varphi_{F_1})}{\nu_r(\varphi_{F_1})}\right)\\[4mm]
\left(\text{respectively } Q_{12}=\left(\dfrac{b\overline{\beta}_{g+1}(\nu_r)}{\nu_r(\varphi_{F_1})}+\theta_2^r(D),0\right)\right)\text{and } Q_{13}=Q_{12} + \left(0, \dfrac{b}{\nu_r(\varphi_{F_1})}\right)\\[4mm]
\end{array}
\end{equation*}
belong to $\Delta_\nu(D)$ when $\theta_2^r(D)\geq 0$ and the point $p_{r+1}\in E_r\cap E_\eta$ is satellite (respectively, free). When $\theta_2^r(D)<0$ and the point $p_{r+1}$ is satellite (respectively, free)), the points in $\Delta_\nu(D)$ are
\begin{equation*}
\begin{array}{c}
Q_{14}=\left(\dfrac{(a+b\delta)\overline{\beta}_{g+1}(\nu_r)}{\nu_r(\varphi_{M_1})},\dfrac{(a+b\delta)\nu_r(\varphi_\eta)}{\nu_r(\varphi_{M_1})}\right)\\[4mm]
\left(\text{respectively, } Q_{14}=\left(\dfrac{(a+b\delta)\overline{\beta}_{g+1}(\nu_r)}{\nu_r(\varphi_{M_1})},0\right)\right), Q_{15}=Q_{14}+\left(0, \dfrac{a+b\delta}{\nu_r(\varphi_{M_1})}\right),\\[4mm]
Q_{16}=\left(\dfrac{a\overline{\beta}_{g+1}(\nu_r)-\theta_2^r(D)\nu_r(\varphi_{M_1})}{\nu_r(\varphi_{M_1})-\delta\nu_r(\varphi_{F_1})},\dfrac{a\nu_r(\varphi_\eta)-\theta_2^r(D)\nu_\eta(\varphi_{M_1})}{\nu_r(\varphi_{M_1})-\delta\nu_r(\varphi_{F_1})}\right)\\[4mm]

\left(\text{respectively, }Q_{16}=\left(\dfrac{a\overline{\beta}_{g+1}(\nu_r)-\theta_2^r(D)\nu_r(\varphi_{M_1})}{\nu_r(\varphi_{M_1})-\delta\nu_r(\varphi_{F_1})},0\right)\right)\\[4mm]
\end{array}
\end{equation*}
\begin{equation*}
\begin{array}{c}
\text{ and } Q_{17}=Q_{16} + \left(0, \dfrac{a}{\nu_r(\varphi_{M_1})-\delta\nu_r(\varphi_{F_1})}\right).\\[4mm]
\end{array}
\end{equation*}

Also, when $p_{r+1}$ is satellite (respectively, free), the point $Q_{18}=(\hat{\mu}_D(\nu_r),\hat{\mu}_D(\nu_\eta))$ (respectively, $Q_{18}=(\hat{\mu}_D(\nu_r),0)$)  belongs to $\Delta_\nu(D)$ by Theorem \ref{Thm_mupico_nonpositive}.

\begin{theorem}\label{Thm_NOBodies_nonspecialcase}
Let $\nu$ be a valuation in Case D. With notations as in the previous paragraphs, the Newton-Okoun\-kov body $\Delta_\nu(D)$ of $D$ with respect to $\nu$ is a quadrilateral if and only if $a\neq0$ and $\theta_2^r(D)\neq 0$. Otherwise, it is a triangle.

The vertices of the quadrilateral are 
\begin{itemize}
\item[(a)] $(0,0),Q_{10},Q_{12}$ (respectively, $Q_{14},Q_{16}$) and $Q_{18}$ when $\theta_2^r(D)> 0$ (respectively, $\theta_2^r(D)<0$), $p_{r+1}$ is the satellite point  $E_\eta\cap E_r$ and $r\not\preccurlyeq \eta.$
\item[(b)] $(0,0),Q_{11},Q_{13}$ (respectively, $Q_{15},Q_{17}$) and $Q_{18}$ when $\theta_2^r(D)> 0$ (respectively, $\theta_2^r(D)<0$), $p_{r+1}$ is the satellite point $E_\eta\cap E_r$ and $r\preccurlyeq \eta.$
\item[(c)] $(0,0),Q_{11},Q_{13}$ (respectively, $Q_{15},Q_{17}$) and $Q_{18}$ when $\theta_2^r(D)> 0$ (respectively, $\theta_2^r(D)<0$) and $p_{r+1}$ is a free point.
\end{itemize}

When $a=0$ and $\theta_2^r(D)<0$, $Q_{16}=Q_{18}=Q_{17}$ and the vertices of the triangle $\Delta_\nu(D)$ are as described in items (a), (b) and (c).

Finally, replacing $\theta_2^r(D)>0$ or $\theta_2^r(D)<0$ with $\theta_2^r(D)=0$ in items (a), (b) and (c) we obtain the vertices of the triangle $\Delta_\nu(D)$ because $Q_{10}=Q_{12}=Q_{14}=Q_{16}$ in Case (a) and $Q_{11}=Q_{13}=Q_{15}=Q_{17}$ otherwise.
\end{theorem}

\begin{proof}
We are going to show that $D^2/2$ is the area of the convex set $\Delta$ generated by the points $(0,0),Q_{14},Q_{15},Q_{16},Q_{17}$ and $Q_{18}$.  The case concerning the points $(0,0),Q_{10},Q_{11},Q_{12},Q_{13}$ and $Q_{18}$ and the fact of being a quadrilateral or a triangle fo\-llow as in the proof of Theorem \ref{Thm_NOBodies_specialcase}.

The area of the triangle  with vertices $(0,0),Q_{14}$ and $Q_{15}$ (respectively, $Q_{16},Q_{17}$ and $Q_{18}$) is 
$$
\frac{(a+b\delta)^2}{2\nu_r(\varphi_{M_1})^2}\overline{\beta}_{g+1}(\nu_r)\Bigg(\text{respectively, }
$$
$$
\left.\dfrac{a}{2(\nu_r(\varphi_{M_1})-\delta\nu_r(\varphi_{F_1}))}\left(\hat{\mu}_D(\nu_r)-\left(\dfrac{a\overline{\beta}_{g+1}(\nu_r)-\theta_2^r(D)\nu_r(\varphi_{M_1})}{\nu_r(\varphi_{M_1})-\delta\nu_r(\varphi_{F_1})}\right)\right)\right).
$$

The area of the trapezium given by $Q_{14},Q_{15},Q_{16}$ and $Q_{17}$ is
$$
\dfrac{-\theta_2^r(D)\left((a+b\delta)(\nu_r(\varphi_{M_1})-\delta\nu_r(\varphi_{F_1}))+a\nu_r(\varphi_{F_1})\right)(\nu_r(\varphi_{M_0})^2-\delta\overline{\beta}_{g+1}(\nu_r))}{2\nu_r(\varphi_{M_1})^2(\nu_r(\varphi_{M_1})-\delta\nu_r(\varphi_{F_1}))^2}.
$$
Summing the above three areas, we notice that the coefficient of $\overline{\beta}_{g+1}(\nu_r)$ vanishes and it suffices to sum the following three fractions  
\begin{equation*}
\begin{array}{c}
\dfrac{a\hat{\mu}_D(\nu_r)}{2(\nu_r(\varphi_{M_1})-\delta\nu_r(\varphi_{F_1}))},\dfrac{a\theta_2^r(D)\nu_r(\varphi_{M_1})}{2(\nu_r(\varphi_{M_1})-\delta\nu_r(\varphi_{F_1}))^2} \text{ and }\\[4mm] 
\dfrac{-\theta_2^r(D)\nu_r(\varphi_{M_1})^2((a+b\delta)(\nu_r(\varphi_{M_1})-\delta\nu_r(\varphi_{F_1}))+a\nu_r(\varphi_{F_1}))}{2\nu_r(\varphi_{M_1})^2(\nu_r(\varphi_{M_1})-\delta\nu_r(\varphi_{F_1}))^2}.
\end{array}
\end{equation*}
After computing, one gets $(2ab+\delta b^2)/2,$ which concludes the proof.  
\end{proof}

\begin{example}
Let $p$ be a general point of the Hirzebruch surface $\mathbb{F}_2$  and $\nu_r$  a non-special divisorial valuation centered at $\mathcal{O}_{\mathbb{F}_2,p},$ whose sequence of maximal contact values is $\{\overline{\beta}_i(\nu_r)\}_{i=0}^3=\{15,51,262,$ $786\}$. Let $\mathcal{C}_{\nu_r}=\{p_i\}_{i=1}^{12}$ (with $p=p_1$) be its  configuration of infinitely near points, $F_1$ the fiber which passes through $p$ and $M_1$ the irreducible section linearly equivalent to $M$ that passes through $p$ and whose strict transform passes through $p_2$ and $p_3$. Notice that this means that the self-intersection of $\tilde{M}_1$ is negative. Then, $\nu_r(\varphi_{F_1})=15$ and $\nu_r(\varphi_{M_1})=45$ and so $2\nu_r(\varphi_{F_1})\nu_r(\varphi_{M_1})-\nu_r(\varphi_{F_1})^2\delta=900>786=\overline{\beta}_{g+1}(\nu_r)$. As a consequence, $\nu_r$ is non-positive at infinity by \cite[Theorem 4.8]{GalMonMor}. 

Let $\nu=\nu_{E_\bullet}$ be the valuation defined by the flag $$E_{\bullet}=\{Z=Z_{12}\supset E_{12}\supset\{p_{13}\}\},$$ where $p_{13}\in E_{9}\cap E_{12}$. By Theorem \ref{Thm_NOBodies_nonspecialcase}, the coordinates of the vertices of the Newton-Okounkov body $\Delta_\nu(2F+5M)$ are 
$$
(0,0),Q_{14}=\left(\dfrac{9432}{45},\dfrac{3132}{45}\right)\!\!,Q_{16}=\left(\dfrac{3597}{15},\dfrac{1197}{15}\right) \text{and } Q_{18}=\left(255,85\right),
$$
since $\nu_r$ is non-minimal with respect to $2F+5M$ by Corollary \ref{Cor_nonminimalresptanyD}, $\theta_2^r(D)<0$ and $12=r\not\preccurlyeq \eta=9.$
\end{example}

Finally, assume that $\nu$ is in Case E. By \cite[Theorem 6.4]{LazMus} and Proposition \ref{Prop_Zardecomp_caso_noespecial}, if $p_{r+1}$ is the satellite point $E_r\cap E_\eta$ and $\theta_2^r(D)\geq 0$ (respectively, $\theta_2^r(D)<0$), the points $Q_{10},Q_{11},Q_{12},Q_{13}$ (respectively, $Q_{14},Q_{15},Q_{16},Q_{17}$) and $Q_{18}$ provided before Theorem \ref{Thm_NOBodies_nonspecialcase} for the satellite case belong to $\Delta_\nu(D)$. When $p_{r+1}$ is a free point and $\theta_2^r(D)<0$ (respectively, $\theta_2^r(D)\geq 0$), the points 
\begin{equation*}
\begin{array}{c}
Q_{14}=\left(\dfrac{(a+b\delta)\overline{\beta}_{g+1}(\nu_r)}{\nu_r(\varphi_{M_1})},0\right), Q_{15}=Q_{14}+\left(0, \dfrac{a+b\delta}{\nu_r(\varphi_{M_1})}\right),\\[4mm]
Q_{16}=\left(\dfrac{a\overline{\beta}_{g+1}(\nu_r)-\theta_2^r(D)\nu_r(\varphi_{M_1})}{\nu_r(\varphi_{M_1})-\delta\nu_r(\varphi_{F_1})},\dfrac{-\theta_2^r(D)}{\nu_r(\varphi_{M_1})-\delta\nu_r(\varphi_{F_1})}\right),\\[4mm]
 Q_{17}=Q_{16} + \left(0, \dfrac{a}{\nu_r(\varphi_{M_1})-\delta\nu_r(\varphi_{F_1})}\right)\\[4mm]
\end{array}
\end{equation*}
(respectively, $Q_{10},Q_{11},Q_{12},Q_{13}$ given before Theorem \ref{Thm_NOBodies_nonspecialcase} for the free case) and $Q_{18}=(\hat{\mu}_D(\nu_r)),b)$ are in $\Delta_\nu(D)$.
\begin{theorem}\label{Thm_NOBcasonoespecial_g*=1M1}
Let $\nu$ be a valuation in Case E. Under the above notations, the Newton-Okounkov body $\Delta_\nu(D)$ of $D$ with respect to $\nu$ is a quadrilateral if and only if $a\neq 0$. Otherwise, it is a triangle.

The vertices of the quadrilateral are 
\begin{itemize}
\item[(a)] $(0,0),Q_{10},Q_{13}$ (respectively, $Q_{15},Q_{16}$) and $Q_{18}$ if $\theta_1^r(D)\geq 0$ (respectively, $\theta_1^r(D)< 0$), $p_{r+1}$ is the satellite point $E_{r}\cap E_\eta$ and $r\not\preccurlyeq \eta.$
\item[(b)]$(0,0),Q_{11},Q_{12}$ (respectively, $Q_{14},Q_{17}$) and $Q_{18}$ if $\theta_1^r(D)\geq 0$ (respectively, $\theta_1^r(D)< 0$), $p_{r+1}$ is the satellite point $E_{r}\cap E_\eta$ and $r\preccurlyeq \eta.$
\item[(c)]$(0,0),Q_{11},Q_{12}$ (respectively, $Q_{14},Q_{17}$) and $Q_{18}$ if $\theta_1^r(D)\geq 0$ (respectively, $\theta_1^r(D)< 0$) and $p_{r+1}$ is a free point.
\end{itemize}
In addition, if $a=0$, then the vertices of the triangle $\Delta_\nu(D)$ are the previous ones where $Q_{16}=Q_{18}=Q_{17}$.
\end{theorem}
\begin{proof}
It follows a  reasoning as in the proof of Theorem \ref{Thm_NOBodies_nonspecialcase} to compute the area of the convex sets  generated by the points given in the statement, and as in Theorem \ref{Cond_Fg*} (b) after taking into account the equalities $$\nu_\eta(\varphi_{M_1})\overline{\beta}_{g^*}(\nu_r)=\nu_r(\varphi_{M_1})\overline{\beta}_{g^*}(\nu_\eta)\text{ and }\nu_\eta(\varphi_{F_1})\overline{\beta}_{0}(\nu_r)=\nu_r(\varphi_{F_1})\overline{\beta}_{0}(\nu_\eta).$$
\end{proof}
%%%%%%%%%%%%%%%%%%%%%%%%%%%%%%%%%%%%%%%%%%%%%%%%%%%%%%%%%%
%%%%%%%%%%%%%%%%% Bibliografía %%%%%%%%%%%%%%%%%%%%%%%%%%%
%%%%%%%%%%%%%%%%%%%%%%%%%%%%%%%%%%%%%%%%%%%%%%%%%%%%%%%%%%

\bibliographystyle{plain}
\bibliography{BIBLIO}

\end{document}